%% file: interaction-flow.tex
\documentclass[11pt]{article}
\usepackage[utf8]{inputenc}
\usepackage[margin=1.1in,a4paper]{geometry}

\usepackage{mathtools}
\usepackage{amsmath, amssymb, graphicx, url, amsthm, subcaption, dsfont}
\usepackage{mathrsfs}
\usepackage[ruled, lined]{algorithm2e}
\usepackage{enumitem}

\usepackage[maxbibnames=99, style=alphabetic, maxalphanames=5]{biblatex}
\usepackage{tikz}
\usepackage{verbatim}
 \usepackage{tcolorbox}
 \newtcolorbox{assbox}{colback=black!5!white,colframe=black!75!black}
  \newtcolorbox{thmbox}{colback=red!5!white,colframe=red!75!black}

\usepackage{amsfonts}       
\usepackage{enumitem}
\usepackage{xcolor}
\usepackage[colorlinks=true,citecolor=blue]{hyperref}    
\usepackage{diagbox}   
\usepackage{tcolorbox}
\usepackage{nicefrac}
\newcommand{\TT}{\mathbb{T}}

\newcommand{\bv}{\boldsymbol{v}}

\newcommand{\bO}{\boldsymbol{O}}

\newcommand{\R}{\mathbb{R}}

\newcommand{\Z}{\mathbb{Z}}
\newcommand{\PP}{\mathbb{P}}
\newcommand{\T}{\mathbb{T}}

\input{shortcuts}

\DeclareMathOperator{\Var}{Var}

\graphicspath{ {./images/} }

\title{
Convergence of drift-diffusion PDEs arising as Wasserstein gradient flows of convex functions

}

\author{
L\'ena\"ic Chizat\thanks{EPFL, \texttt{lenaic.chizat@epfl.ch}},\; Maria Colombo\thanks{EPFL, \texttt{maria.colombo@epfl.ch}},\; Xavier Fern\'andez-Real\thanks{EPFL, \texttt{xavier.fernandez-real@epfl.ch}}
}

\date{}

\begin{document}
\maketitle

\begin{abstract}
We study the quantitative convergence of drift-diffusion PDEs that arise as Wasserstein gradient flows of linearly convex functions over the space of probability measures on~$\RR^d$. In this setting, the objective is in general not \emph{displacement} convex, so it is not clear a priori whether global convergence even holds. Still, our analysis reveals that diffusion {allows} a favorable interaction between Wasserstein geometry and linear convexity, leading to a general quantitative convergence theory, analogous to that of gradient flows in convex settings in the Euclidean space. 

Specifically, we prove that if the objective is convex and suitably coercive, the suboptimality gap decreases at a rate $O(1/t)$. This improves to a rate faster than any polynomial---or even exponential in compact settings---when the objective is strongly convex relative to the entropy. 

 Our results extend the range of mean-field Langevin dynamics that enjoy quantitative convergence guarantees, and enable new applications to optimization over the space of probability measures. To illustrate this, we show quantitative convergence results for the minimization of entropy-regularized nonconvex problems, we propose and study an \emph{approximate Fisher Information} regularization covered by our setting, and we apply our results to an estimator for trajectory inference which involves the minimization of the relative entropy with respect to the Wiener measure in path space.
\end{abstract}

\section{Introduction}\label{sec:introduction}

Wasserstein gradient flows (WGFs) are dynamics in the space of probability measures that can be interpreted as steepest descent curves for an objective function under the Wasserstein metric. They were first introduced in~\cite{jordan1998variational, otto2001geometry} and rapidly emerged as an important tool to study the well-posedness of certain evolution equations from mathematical physics, in particular in low-regularity or infinite-dimensional settings~\cite{ambrosio2005gradient}. More recently, WGFs have become a source of inspiration for the design and analysis of optimization algorithms in the space of probability measures, such as sampling~\cite{wibisono2018sampling, chewi2023optimization}, variational inference~\cite{lambert2022variational}, neural networks optimization~\cite{mei2019mean, chizat2018global, rotskoff2022trainability, sirignano2020mean, fernandezreal2022continuous}, and various other applications in data science~\cite[Chap.~6]{chewi2024statistical}. Although WGFs generally require time and/or particle discretization to be implemented in practice, their analysis already offers valuable insights and guidelines for algorithm design.

A key property in the analysis of the long-time behavior of WGFs is the notion of displacement convexity~\cite{mccann1997convexity}, which refers to convexity along Wasserstein geodesics. When the objective function is (strongly) displacement convex, WGFs enjoy the quantitative convergence guarantees of gradient flows of (strongly) convex functions in the Euclidean space. However, displacement convexity is a restrictive assumption which is rarely met in data science applications; for instance, in the context of sampling, it corresponds to assuming log-concavity of the target distribution. It is also difficult to verify in practice, as it relies on properties of the data that are usually unknown. In contrast, the property of (linear) convexity is a standard modelling principle in data science and allows to cover a wide range of applications, some of them already mentioned (other examples are given in Section~\ref{sec:applications}).  Unfortunately, under this assumption, there is a mismatch between the geometry of the objective and that of the optimization dynamics: it is not clear a priori whether WGFs can benefit at all from the linear convexity of the objective function. As a matter of fact, there are many examples of WGFs of convex objectives with non-optimal stable stationary points, so it is clear that additional structure must be present for global convergence guarantees to hold. In this paper, we investigate such a favorable structure, consisting of WGFs with a diffusive term.

\subsection{Setting and context}
Let $(\Omega,\dist)$ be the ambient space, which is assumed to be either $\RR^d$, the $d$-torus $\TT^d=(\RR/\ZZ)^d$, or the closure of an open convex domain throughout the paper.
 Let $\Pp_2(\Omega)$ be the set of probability measures with finite second moment endowed with the topology generated by the $L^2$-Wasserstein distance $W_2$; see ~\cite[Chap.~6]{villani2008optimal} or \cite[Chap.~3]{figalli2021optimal}. We are interested in the behavior of a WGFs for an objective function $\Ff:\Pp_2(\Omega)\to \RR\cup \{+\infty\}$ of the form
\begin{align}\label{eq:objective}
\Ff(\mu)=\Gg(\mu) + \tau \Hh(\mu)
\end{align}
where  $\tau>0$ is the diffusivity strength and where:
\begin{itemize}
\item $\Hh$ is the negative differential entropy, defined as
\begin{align}
\Hh(\mu) = \begin{cases}
\int \mu \log \mu &\text{if $\mu$ absolutely continuous,}\\
+\infty &\text{otherwise.}
\end{cases}
\end{align}
\item $\Gg:\Pp_2(\Omega)\to \RR$ is a function that satisfies Assumption~\ref{ass:regularity-G}  below. 
\end{itemize}

Observe that the first-variation of $\Ff$ at an absolutely continuous measure $\mu$ is equal to (whenever it is well-defined)
\begin{align}\label{eq:variation-F}
\Ff'[\mu]=\Gg'[\mu]+\tau \log(\mu).
\end{align}
{Here, $\Gg'[\mu]$ denotes the first variation in $L^2$ of the functional $\Gg$, namely  $\lim_{t \downarrow 0} \frac{\Gg(\mu+t \sigma)- \Gg(\mu)}{t} = \int \Gg'[\mu] \, d \sigma$  for every $ \sigma$ such that $\mu + t\sigma \in \Pp(\Omega)\cap {\rm dom}(\Ff)$ for $t > 0$ small enough. 
At a formal and general level, a WGF of a function $\Ff$ is a curve $(\mu_t)_{t\geq 0}$ in $\Pp_2(\Omega)$ that solves $\partial_t \mu_t ={\rm div}\, (\mu_t \nabla \Ff'[\mu_t])$. By plugging \eqref{eq:variation-F}, we get the following definition for WGFs in our context:

\begin{definition}[Wasserstein gradient flows of $\Ff$]
\label{defi:mut}
A WGF $(\mu_t)_{t\geq 0}$ of $\Ff$ is a continuous curve $\RR_+\to \Pp_2(\Omega)$ that is a distributional solution (with Neumann boundary conditions if $\Omega$ has boundaries) of 
\begin{align}\label{eq:PDE}
\partial_t \mu_t = {\rm div}\, (\mu_t \nabla \Gg'[\mu_t]) +\tau \Delta \mu_t,
\end{align}
it has a positive density for $t>0$, and satisfies, for all $0\leq t_1<t_2<+\infty$,
\begin{equation}
    \label{eqn:en-dissip}
    \Ff(\mu_{t_2})-\Ff(\mu_{t_1}) = -\int_{t_1}^{t_2} \Vert \nabla \Ff'[\mu_s]\Vert^2_{L^2(\mu_s)}\d s = -\int_{t_1}^{t_2} \int \vert \nabla \Gg'[\mu_s] +\tau \nabla \log (\mu_s)\vert_2^2\d\mu_s\d s.
\end{equation}
\end{definition}

For our long-time convergence analysis, we assume the following regularity on $\Gg$: 

\begin{assumption}[Regularity of $\Gg$]\label{ass:regularity-G} 
For all $\mu\in \Pp_2(\Omega)$, $\Gg$ admits a first-variation at $\mu$ (\cite[Chap.~7]{santambrogio2015optimal}), denoted $\Gg'[\mu]$, which belongs to $\Cc^1(\Omega;\RR)$ (it is then unique up to an additive constant), and satisfies:
\begin{enumerate}[label=(\roman*)]
\item \label{it:Ai} If $\Omega = \T^d$, $|\nabla \Gg'[\mu]|\le L$ in $\Omega$ for some $L > 0$, and for all $\mu\in \Pp_2(\Omega)$.
\item  \label{it:Aii}  If $\Omega = \R^d$, $\Gg'[\mu](x)=\frac{\alpha}2 \vert x\vert^2+V[\mu](x)$ with $|\nabla V[\mu]|\le L$ in $\Omega$  for some $\alpha, L>0$, and  for all $\mu\in \Pp_2(\Omega)$.
\end{enumerate}
\end{assumption}
Some examples of functions $\Gg$ that satisfy Assumption~\ref{ass:regularity-G} arising in applications are discussed in Section~\ref{sec:applications}. 

The global-in-time existence of the WGF is satisfied in general under fairly general assumptions, but slightly more restricitive than Assumption~\ref{ass:regularity-G} in terms of regularity: for instance, Assumption~\ref{ass:regularity-G-W1} below directly ensures that the functional is displacement semiconvex and imply the existence of a WGF.

}

\subsection{Presentation of our results}
Our goal in this paper is to study the (quantitative) convergence of $\Ff(\mu_t)$ (of the form \eqref{eq:objective}) towards $\inf \Ff$ under a (linear) convexity assumption on $\Ff$. Thus, throughout this work, we will assume that:
\begin{equation}
    \label{eq:convexity-F}
    \text{The objective $\Ff$ is convex.}
\end{equation}

We notice that, under Assumption~\ref{ass:regularity-G} and convexity, the functional $\Ff$ is lower bounded and admits a minimizer:

\begin{remark}[$\Ff$ is lower bounded and has a minimizer]
\label{rem:assB}
By convexity, \eqref{eq:convexity-F}, we know 
\begin{equation}
    \label{eq:convexity_F_F'}
    \Ff(\mu) \ge \Ff(\nu) + \int_\Omega \Ff'[\nu](\mu-\nu)\qquad\text{for all}\quad \mu, \nu \in \Pp(\Omega).
\end{equation}
Let $\mu_0 =  c e^{-|x|}\, \d x \in \Pp_2(\Omega)$. By Assumption~\ref{ass:regularity-G}, we have $\Ff'[\mu_0](x) \ge A_{\mu_0} |x|^2 - B_{\mu_0}$ in $\Omega$ (for some $A_{\mu_0}, B_{\mu_0}$ that depend also on $\tau$). Then, $\Ff$ is lower bounded,
\[
\Ff(\mu) \ge \Ff(\mu_0) -\int_\Omega \Ff'[\mu_0]\mu_0 + \int_\Omega (A_{\mu_0}|x|^2 - B_{\mu_0})\d\mu(x) \ge -C(\mu_0) > -\infty.
\]
A similar argument shows the tightness of $\{\Ff\le \alpha\}$.  Combined with the lower semi-continuity of the functional (by  Assumption~\ref{ass:regularity-G} and  $W_1/{\rm Lip}$ duality), gives the existence of a minimizer. In particular, by \eqref{eq:convexity_F_F'}, if $\Ff'[\bar\mu]\equiv c$ is constant in $\Omega$ for some $\bar \mu\in \Pp(\Omega)$, then $\bar\mu$ is a minimizer. 
\end{remark}

We also remark that, under a stronger version of Assumption~\ref{ass:regularity-G} (namely,  Assumption~\ref{ass:regularity-G-W1}) when $\Omega$ is compact, we can always make $\Ff$ convex in \eqref{eq:objective} by taking $\tau$ large enough (see Proposition~\ref{prop:reg-to-convex}).

In order to quantify the convexity of $\Ff$, we define the \emph{critical diffusivity} $\tau_c$ as
\begin{equation}
\label{eq:tau_c}
\tau_c \coloneqq \inf\{ \tau' \geq 0\; :\; \Gg+\tau' \Hh \text{ is convex}\}.
\end{equation}
In other words, when $\tau_c<+\infty$, we have that $\Gg$ is $(-\tau_c)$-convex\footnote{Following e.g.~\cite{lu2018relatively}, we say that $\Ff$ is $\alpha$-convex relative to $\Hh$ if $\Ff-\alpha\Hh$ is convex. When $\alpha>0$ this means that $\Ff$ is $\alpha$-strongly convex relative to $\Hh$, while if $\alpha<0$ then $\Ff$ is only semiconvex.}  relative to $\Hh$. Since the infimum defining $\tau_c$ is attained, the convexity assumption on $\Ff$, \eqref{eq:convexity-F}, amounts to requiring $\tau\geq \tau_c$. Our main result is then the following:

\begin{theorem}[Main result]\label{thm:main}
Let $\Omega=\RR^d$, let $L, M_0, \alpha, \tau > 0$. Assume that $\Ff=\Gg+\tau\Hh$ is convex and  $\Gg$ satisfies Assumption~\ref{ass:regularity-G} (with $L > 0$ and $\alpha > 0$), let $\tau_c$ be given by~\eqref{eq:tau_c}, and let $\mu_t$ be a WGF~\eqref{eq:PDE} of $\Ff$ starting from $\mu_0\in \Pp_2(\Omega)$. 

Suppose, also, that the initialization $\mu_0$ is $M_0$-subgaussian, i.e., $\int e^{ \vert x\vert^2/M^2_0}\d \mu_0(x)\leq 2$. Then:
\begin{enumerate}
\item If $\tau\geq \tau_c>0$ then there exists an explicit $C>0$,  depending only on $d$, $\alpha$, $\tau$, $L$ and $M_0$, such that, for $t\geq  t_0= \alpha^{-1}(5+\log(M_0^2\alpha/\tau))$, it holds
$$
\Ff(\mu_t)-\inf \Ff\leq \Big((\Ff(\mu_{t_0})-\inf \Ff)^{-1} +C(t-t_0) \Big)^{-1}.
$$
\item If $\tau>\tau_c\geq 0$ then for any $\kappa>1$, there exists an explicit $C>0$, depending only on $d$, {$\kappa$}, $\alpha$, $\tau$,  $L$, and $M_0$, such that, for $t\geq t_0=\alpha^{-1}(5+\log(M_0^2\alpha\kappa/\tau))$, it holds
$$
\Ff(\mu_t)-\inf \Ff\le \Big((\Ff(\mu_{t_0})-\inf \Ff)^{-1/\kappa} +C (\tau - \tau_c)^{1+1/\kappa}(t-t_0) \Big)^{-\kappa}.
$$
\end{enumerate}
\end{theorem}
\begin{remark}
    While we do not track the exact value of the constants, one can see from our proofs that they are explicit and can be computed (see, e.g., the proof of Proposition~\ref{prop:two_sided_global}). 
\end{remark}

We emphasize that the {sub}linear $O(1/t)$ convergence rate in Point 1 holds under mere convexity of $\Ff$, that is, $\tau=\tau_c$, without requiring displacement convexity or strong convexity. This highlights the role of diffusion in enabling convergence in a setting where classical Wasserstein gradient flow theory provides no guarantees.

A key technical step in the proof consists in deriving uniform-in-time two-sided Gaussian estimates for the density of $\mu_t$ (with explicit computable constants). More specifically, we prove (see Proposition~\ref{prop:bounds}) that for all $0<\eps<1$, there exists $c_\eps>0$  such that for all $ t\geq \alpha^{-1}(4 + \frac12\log(1+1/\eps))$ it holds
$$
c_\eps e^{-\frac{\alpha(1+\eps)\vert x\vert^2}{2\tau}} \leq \mu_t(x) \leq c_\eps^{-1}e^{-\frac{\alpha(1-\eps)\vert x\vert^2}{2\tau}}.
$$
Once such estimates are established, Theorem~\ref{thm:main} follows by arguments from convex optimization, where the density estimates are used to convert the $L^2$-gradient estimates into $W_2$-gradient estimates via the Poincaré inequality. 

 More precisely, in the strongly convex case we rely on the \L{}ojasiewicz inequality of exponent $\frac{1}{2}$--- a.k.a.~Polyak--\L{}ojasiewicz inequality~\cite{rebjock2024fast}---in the $L^2$ geometry, satisfied by $\Ff$ on a suitable subset of the probability densities (see Lemma~\ref{lem:PL}). One can find such an argument already used implicitly in the proof of exponential convergence of ultrafast diffusions on a compact domain~\cite{caglioti2015gradient, iacobelli2019weighted}. In the simply convex case, instead, our argument relies on the \L{}ojasiewicz inequality of exponent $0$ in the $L^2$ geometry, which is the typical proof scheme for the convergence of gradient flows of smooth and convex functions in Euclidean setting (see e.g.~\cite[Theorem 3.3]{bubeck2015convex}). This argument was already used in~\cite[Proposition~7]{arbel2019maximum} to study the WGF of \emph{kernel means discrepancies}~\cite[Sec.~12.6]{wainwright2019high}. (However, we note that in their setting, nothing indicates that the density estimates are satisfied in general and, therefore, that one can obtain convergence via such an argument.) In both cases, we use the density estimates to convert the \L{}ojasiewicz inequalities in the $L^2$ geometry into \L{}ojasiewicz inequalities in the Wasserstein geometry~\cite{blanchet2018family}. 

 \medskip
 
 We also study the compact case of the $d$-torus $\TT^d$. In that case, we in fact obtain exponential (rather than superpolynomial) convergence in the case $\tau>\tau_c$, and the rate is linear in $\tau-\tau_c$.   Although previous works already cover the compact case when $\tau>\tau_c$, our estimate’s linear dependence on $\tau-\tau_c$ represents an exponential improvement over those results.
 
\begin{theorem}\label{thm:compact}
Let $\Omega=\TT^d$, $L, \tau > 0$, and assume that $\Ff=\Gg+\tau\Hh$ is convex and $\Gg$ satisfies Assumption~\ref{ass:regularity-G}  (with $L  > \tau$), let $\tau_c$ be given by~\eqref{eq:tau_c}, and let $\mu_t$ be a WGF~\eqref{eq:PDE} starting from $\mu_0\in \Pp_2(\Omega)$. Then:
\begin{enumerate}
\item  If $\tau>\tau_c\geq 0$, then for $t\geq t_0= \tau/(4L^2)$ it holds
\begin{equation}
\label{ts:expconv}
\Ff(\mu_t)- \inf \Ff\leq \exp\Bigl(-c_1(\tau-\tau_c)(t-t_0)\Bigr)(\Ff(\mu_0)-\inf \Ff).
\end{equation}
\item If $\tau\geq \tau_c>0$, then for $t\geq  t_0=\tau/(4L^2)$ it holds
$$
\Ff(\mu_t)-\inf \Ff \leq \Big((\Ff(\mu_0)-\inf \Ff)^{-1} +c_2 (t-t_0) \Big)^{-1}.
$$
\end{enumerate}
The constants $c_1$ and $c_2$ depend only on $d$ and $L/\tau$. 
\end{theorem}
\begin{remark}[Explicit constants]
We can actually take:
\[
c_1 = \frac{8m\pi^2}{M}\qquad\text{and}\qquad c_2 = \frac{m\pi^2}{M^2},\qquad\text{with} \quad \begin{cases} M
= 4\cdot (2\sqrt{2})^d (\tfrac{{L}}{\tau})^{d} \\
 m=5^{-1} e^{-\frac38 ({L}/\tau)^2 d} (\tfrac{L}{\tau})^d (\tfrac{\sqrt{2}}{3})^d.
 \end{cases}
\]
\end{remark}
Our approach in the compact case can be adapted to compact manifolds or domains (in which case the PDE~\eqref{eq:PDE} should be considered with Neumann boundary conditions). Indeed, the key property that the domain should satisfy is the existence of suitable two-sided bounds on the transition kernels of uniformly elliptic diffusion processes, a topic with a rich literature~\cite{sheu1991some, wang1997sharp}. We focus on the torus $\TT^d$ because, in this setting, one can explicitly track the dependence of these kernel bounds on the parameters of the problem (see Lemma~\ref{lem:density-bounds-torus}).

 Finally, let us mention that in order to obtain our results, we derive explicit upper and lower bounds for transition kernels for the Fokker-Planck equation (in $\T^d$, and with confinement in $\R^d$) with a merely bounded drift (see Proposition~\ref{prop:heat_kernel_Rd}, Corollary~\ref{cor:heat_kernel_Td}, Proposition~\ref{prop:two_sided_global}). While such estimates are classical at the qualitative level (see, e.g.~\cite{menozzi2021kernel}), we could not find them derived with explicit constants in the literature.

\subsection{Comparison with previous literature and applications}
Many of the functions whose WGF have been studied in the literature happen to be convex. This is the case, for instance, of the relative entropy, whose WGF yields the Fokker-Planck equation~\cite{jordan1998variational}, but also more general internal energies (yielding the porous medium equation~\cite{otto2001geometry} or the fast diffusion equation~\cite{iacobelli2019weighted}), the Fisher Information~\cite{gianazza2009wasserstein}, or kernel mean discrepancies with continuous or singular kernels~\cite{boufadene2023global, altekruger2023neural, hagemann2024posterior, de2023sharp}. Long-time convergence to minimizers has been obtained in many of these settings, with arguments that are generally tailored to each specific case. In~\cite{carrillo2020long}, the property of convexity of the objective was identified to play a key role for the long-time behavior in the case where the objective is the sum of a pairwise interaction energy and entropy. 

A line of research has emerged recently to study WGF of more generic convex functions regularized by entropy, referred as \emph{mean-field Langevin} dynamics~\cite{hu2021mean}. This term usually refers to WGF of objectives of the form $\Ff=\Gg+\tau\Hh$ where the function $\Gg$ is convex (while we do not make such an assumption in general). This setting was originally motivated
by the analysis of the dynamics of infinite-width two-layer neural networks. In this case, the  function $\Gg$ can be assumed of the form
$$
\Gg(\mu) = R\left(\int \Phi \d\mu\right),
$$
where $\Phi\in \Cc^{1,1}(\RR^d;E)$, $R\in \Cc^{1,1}(E;\RR)$, $R$ is convex, and $E$ is a Hilbert space. This defines a function $\Gg$ that is linearly convex but not displacement convex in general. The study of WGFs under linear convexity turned out useful to study other applications, such as the approximate computation of Wasserstein barycenters~\cite{chizat2023doubly}, sampling techniques such as Adaptive Biasing Force~\cite{lelievre2025convergence}, or trajectory inference~\cite{chizat2022trajectory}. 
 For all these applications, one may wonder why even considering WGFs, since more standard convex optimization methods, such as mirror descent and variants~\cite{nemirovskij1983problem}, appear better suited to exploit the linear convexity of $\Ff$. But although these methods enjoy rapid convergence rates at the continuous level, their implementation  requires to discretize the domain on a fixed grid (or analogous procedures), which quickly becomes intractable as $d$ increases. In contrast, WGF can often be discretized efficiently even with $d$ large\footnote{At least on bounded time intervals. This is the case for instance if $\mu\mapsto \nabla \Gg'[\mu]$ is Lipschitz continuous from $H^{-s}$ to $L^\infty$ for some $s>d/2$ (used implicitly in~\cite{mei2019mean}), a property which can be enforced by design in data science applications.}  via interacting particle systems. As a result, while both discretization and long-time convergence of WGFs present theoretical challenges, in practice, the main bottleneck typically lies in the long-time behavior, which is the primary focus of this work.

Global convergence of WGFs when $\Gg$ is convex was shown in~\cite{mei2018mean, hu2021mean} and quantitative rates of convergence (measured in terms of the suboptimality gap, or the relative entropy) were obtained~\cite{nitanda2022convex,chizat2022mean} and where improved to $L^p$ convergence in~\cite{chen2022uniform}. The latter work also shows uniform propagation of chaos, see also~\cite{suzuki2023convergence, nitanda2024improved, kook2024sampling, chewi2024uniform, nitanda2024improved} on this topic. 

Let us state the quantitative convergence theory in this case, as a detailed comparison with our result is instructive.  In the aforementioned works, \cite{nitanda2022convex, chizat2022mean}, the authors consider an objective $\Ff=\Gg+\tau\Hh$ where $\Gg$ is \emph{convex}. In addition, they assume that the \emph{proximal Gibbs probability measures}
\begin{equation}\label{eq:prox-gibbs}
    \nu_t \propto e^{-\Gg'[\mu_t]/\tau}
\end{equation}
satisfy a log-Sobolev inequality (LSI) uniformly in $t$, with constant $C_{\rm {LSI}}(\tau)$; i.e., 
\begin{align}\label{eq:LSI}
\Hh(\mu|\nu_t)\leq  C_{\rm{LSI}}(\tau)  \int \Big\vert \nabla \log \frac{\d\mu}{\d \nu_t}\Big\vert_2^2\d\mu\qquad \text{for any}\quad \mu\in \Pp_2(\Omega), \ \mu \ll \nu_t,\ t > 0,
\end{align}
where $\Hh(\mu|\nu) = \int \log(\d\mu/\d\nu)\d\mu$ is the relative entropy of $\mu$ with respect to $\nu$ (a.k.a.~Kullback-Leibler divergence). Then~\cite{nitanda2022convex} and~\cite{chizat2022mean} obtain the convergence guarantee
\begin{align}\label{eq:rate-LSI}
\Ff(\mu_t) - \inf \Ff \leq \exp(-\tau \cdot C_{\rm{LSI}}(\tau)^{-1} \cdot t)\cdot  (\Ff(\mu_0)-\inf \Ff).
\end{align}
Typically, $x\mapsto \Gg'[\mu_t](x)$ is \emph{not} a convex function\footnote{If $\Gg$ is linearly convex and $\Gg'[\mu]$ is a convex function on $\Omega$ for all $\mu\in \Pp_2(\Omega)$ then $\Gg$ is displacement convex~\cite[Prop.~B.6]{gangbo2022global}. Our focus is rather on the \emph{non} displacement convex setting.} and $C_{\rm{LSI}}(\tau)$, when it exists, diverges exponentially as $\tau\to 0$. Remark also that Assumption~\ref{ass:regularity-G} implies that $C_{\rm{LSI}}(\tau)<+\infty$, by the Lipschitz perturbation criterion~\cite{aida1994logarithmic}.

In this work, we remove the convexity assumption on $\Gg$, and instead only make a convexity assumption on $\Ff$. In other words, previous works correspond to $\tau>\tau_c=0$ (with $\tau_c$ given by~\eqref{eq:tau_c}) while our main result deals with the more general case where $\tau\geq \tau_c\geq 0$ and $\tau>0$. This more general setting allows to unlock a new set of applications for the mean-field Langevin framework, such as: 
\begin{itemize}
\item the regularization of any function $\Gg$ satisfying Assumption~\ref{ass:regularity-G} on a compact domain (Proposition~\ref{prop:reg-to-convex}). This includes for instance functions that include nonconvex interaction energies (Corollary~\ref{cor:interactions}).
    \item a greater flexibility in the choice of regularization for convex objectives $\Gg$ beyond entropy. In Section~\ref{sec:AFI}, we propose an approximate Fisher-Information regularization that is covered by our convergence guarantees;
    \item this structure also arises naturally when one considers objectives derived from the relative entropy with respect to the Wiener measure in path space. This setting, with applications to trajectory inference, is discussed  in Section~\ref{sec:trajectory-inference}.
\end{itemize} 
 
Another advantage of the approach in this paper regards the quantification of the rate of convergence in terms of the strong convexity constant $\tau-\tau_c$. Indeed, in the compact and strongly convex case $\tau>\tau_c$, we obtain exponential convergence \eqref{ts:expconv}  with a rate that is linear in  the strong convexity constant $\tau-\tau_c$, when $\tau_c>0$. This represents an exponential improvement for the dependency on $\tau - \tau_c$ over the rate that can be obtained via an application of \eqref{eq:rate-LSI} 
to $\Gg+\tau_c\Hh$ (which is convex) in place of $\Gg$. Indeed, the convergence rate provided by \eqref{eq:rate-LSI} is driven by the LSI constant of the probability measures
$$
\nu_t \propto e^{-\frac{\Gg'[\mu_t]+\tau_c \log(\mu_t)}{\tau-\tau_c}} = \mu_t^{-\frac{\tau_c}{\tau-\tau_c}}\cdot e^{-\frac{\Gg'[\mu_t]}{\tau-\tau_c}},
$$
which in general degrades exponentially with $\tau-\tau_c$, even in presence of adequate positive two-sided density bounds on $\mu_t$ (because of the second factor).

\subsection{Organization of the paper}
In the next section (Section~\ref{sec:compact}), we prove Theorem~\ref{thm:compact} on the $d$-torus $\TT^d$. This case serves as a warm-up to introduce the general proof mechanisms. The main result is proved in Section~\ref{sec:mainproof}, while the two-sided Gaussian estimates needed in the proof are postponed to Section~\ref{sec:gaussian}. Finally, applications to optimization methods are discussed in Section~\ref{sec:applications}. Therein, we discuss two regularization terms that enable global convergence of WGFs: entropy and approximate Fisher Information. We also present a setting motivated by the problem of trajectory inference, where our results apply without the need for extra regularization.

\paragraph{Notation} We often identify measures with their Lebesgue densities, which should not lead to confusion since we mostly deal with absolutely continuous measures. The $L^p$-Lebesgue spaces are denoted $L^p(\lambda)$ for a reference measure $\lambda$ or simply $L^p$ when the reference measure is Lebesgue. We also write $\Vert \cdot \Vert_{{\rm TV}}$ for the total-variation norm on the space of finite signed measures and for a function $f:\Xx\to \RR$, we define its oscillation semi-norm as $\Vert f\Vert_{\rm osc}=\sup_{x\in \Xx} f(x)-\inf_{x\in \Xx} f(x)$. Finally the variance of a function $f$ under the probability measure $\lambda$ is denoted by $\Var_\lambda(f) = \Vert f-m\Vert^2_{L^2(\lambda)}$ where $m = \int f\d\lambda$. We reserve the notations $\Vert \cdot \Vert$ for functional norms and use $\vert \cdot\vert$ for the $\ell_2$ norm of vectors and the modulus of complex numbers. For $\mu,\nu$ probability measures, we write $\Hh(\mu)$ for the negative differential entropy of $\mu$ and $\Hh(\mu|\nu)$ for the entropy of $\mu$ relative to $\nu$.

\section{The compact case: proof of Theorem~\ref{thm:compact}}\label{sec:compact}

In order to show Theorem~\ref{thm:compact}, we proceed by first proving several preliminary lemmas that will be useful in the proof.

The first lemma shows a Polyak--\L{}ojasiewicz (P\L{}) inequality in the $L^2$ geometry for the strongly convex case. We state this result at a higher level of generality than needed in this section because we will re-use it in the noncompact case (where the reference measure $\lambda$ will be Gaussian).

 \begin{lemma}[P\L{} inequality]\label{lem:PL}
 Assume that $\Ff=\Gg+\tau\Hh$ is convex and that $\tau>\tau_c$ with $\tau_c$ given by \eqref{eq:tau_c}. 
 
Consider $\lambda\in L^1_+(\Omega)$ such that $\Hh(\lambda)<+\infty$ and let $\Pp_{\leq M}(\Omega)$ be the set of probability densities in $\Pp_2(\Omega)$ bounded by $M\lambda$.   Then, $\Ff$ is $(\tau-\tau_c)/M$-strongly convex relative to the $L^2(\lambda)$-norm over $\Pp_{\leq M}(\Omega)$.
 
 Moreover, let $\mu, \mu^*\in \Pp_2(\Omega)$ such that $\Ff'[\mu]\in L^2(\lambda)$ and for some  $M>0$, $\mu,\mu^*\in \Pp_{\leq M}(\Omega)$. Then, it holds
  \begin{align}\label{eq:PL-1}
  \Ff(\mu) - \Ff(\mu^*)\leq \frac{M}{2(\tau-\tau_c)} \Var_\lambda(\Ff'[\mu]).
  \end{align}
  \end{lemma}
  \begin{proof}
  Consider the function $\varphi(s)=s\log(s)$ for $s\in \RR_+$. Since $\varphi''(s)=1/s$ for $s>0$, this function is $(1/M)$-strongly convex on the interval $[0,M]$.
Therefore, we have that $\mu\mapsto \Hh(\mu|\lambda)=\int \varphi(\d\mu/\d\lambda)\d\lambda$  is $(1/M)$-strongly convex relative to the $L^2(\lambda)$ norm over the (linearly convex) set $\Pp_{\leq M}(\Omega)$. Explicitly, for $\mu,\mu'\in \Pp_{\leq M}(\Omega)$, it holds
$$
\Hh(\mu'|\lambda)\geq \Hh(\mu|\lambda) + \int \log\frac{\d \mu}{\d \lambda} \d (\mu'-\mu)+\frac{1}{2M}\int \left(\frac{\d\mu}{\d\lambda}-\frac{\d\mu'}{\d\lambda}\right)^2\d\lambda.
$$
Moreover, $\Hh(\mu)= \Hh(\mu|\lambda)+\int \log(\lambda)\d \mu$ and this last term is integrable because $\mu\leq M\lambda$ and we have assumed $\Hh(\lambda)<+\infty$. Since $\Hh$ only differs from $\Hh(\cdot|\lambda)$ by a linear term, it satisfies the same strong convexity property.

Now, since $\Gg+\tau_c\Hh$ is convex, we have that $\Ff=\Gg+\tau \Hh$ is $\beta$-strongly convex over  $\Pp_{\leq M}(\Omega)$ with $\beta\coloneqq(\tau-\tau_c)/M$. In particular, letting $(\rho,\rho^*)=(\frac{\d\mu}{\d \lambda},\frac{\d \mu^*}{\d\lambda})$ we have
 $$
 \Ff(\mu^*)\geq \Ff(\mu)+\int \Ff'[\mu]\d( \mu^*-\mu) + \frac{\beta}{2}\Vert \rho-\rho^*\Vert^2_{L^2(\lambda)}.
 $$
 By Cauchy-Schwartz's inequality in $L^2(\lambda)$ and since $\mu-\mu^*$ has zero total mass, we also have 
 $$\int \Ff'[\mu]\d (\mu-\mu^*) \leq \sqrt{\Var_\lambda(\Ff'[\mu])}\Vert \rho-\rho^*\Vert_{L^2(\lambda)} \leq \frac1{2\beta}\Var_\lambda(\Ff'[\mu]) + \frac{\beta}{2}\Vert \rho-\rho^*\Vert^2_{L^2(\lambda)}
 .$$ %
Combining it with the previous inequality, this yields~\eqref{eq:PL-1}.
  \end{proof}

 The next lemma basically shows Theorem~\ref{thm:compact} assuming universal lower and upper bounds for $\mu_t$.

\begin{lemma}\label{lem:compact-convergence}
Assume that $\Omega= \TT^d$  (or more generally, a bounded convex domain)  and that $\Ff=\Gg+\tau\Hh$ is convex and  $\Gg$  satisfies Assumption~\ref{ass:regularity-G}. Let $\mu_t$ be a WGF~\eqref{eq:PDE} starting from $\mu_0\in \Pp_2(\Omega)$.
Assume that there exists $m,M>0$ such that $m\leq \mu_t\leq M$, $\forall t\geq 0$, and denote by $C_P>0$ the Poincaré constant of $\Omega$. Recall that $\tau_c$ is given by~\eqref{eq:tau_c}. It holds:
\begin{enumerate}
\item If $\tau>\tau_c$, then 
\begin{align}
\Ff(\mu_t)-\inf \Ff \leq \exp\Bigl(-\frac{2(\tau-\tau_c)mt}{MC_P^2}\Bigr)(\Ff(\mu_0)-\inf \Ff).
\end{align}
\item If $\tau=\tau_c$, then  
\begin{align}
\Ff(\mu_t)-\inf \Ff \leq \Big((\Ff(\mu_0)-\inf \Ff)^{-1} +\frac{mt}{4M^2|\Omega|C_P^2} \Big)^{-1}
\end{align}
where $|\Omega|$ denotes the Lebesgue measure of $\Omega$.
\end{enumerate} 
\end{lemma}

\begin{proof}
 We start by observing that, by the bounds on $\mu_t$, $(\mu_t)_t$ is (weakly) converging  up to subsequences to some $\mu^*$ with the same upper and lower bounds. Moreover, $\mu^*$ is stationary (in the sense that it has no energy dissipation, \eqref{eqn:en-dissip}), and therefore, thanks to its positive lower bound, $\nabla \Ff'(\mu^*) \equiv 0$ in $\Omega$; i.e., since $\Ff$ is linearly convex, $\mu^*$ is necessarily a minimizer (see Remark~\ref{rem:assB}). 

Let $\lambda$ be the Lebesgue measure on $\Omega$. By Lemma~\ref{lem:PL}, $\Ff$ satisfies
$$
\Ff(\mu)-\Ff(\mu^*)\leq \frac{M}{2(\tau-\tau_c)} \Var_\lambda\!\left( \Ff'[\mu]\right)
$$
for all $\mu\leq M$. Next, using the Poincaré inequality on $\Omega$ and the lower bound on $\mu_t$ we also have for $t\geq 0$  
$$
\Var_\lambda\!\left( \Ff'[\mu_t]\right)\! \leq C_P^2 \big\Vert \nabla \Ff'[\mu_t]\big\Vert^2_{L^2(\lambda)}\leq \frac{C_P^2}{m} \big\Vert \nabla \Ff'[\mu_t]\big\Vert^2_{L^2(\mu_t)}.
$$
All in all, we obtain  for $t\geq 0$,
$$
\Ff(\mu_t)-\Ff(\mu^*)\leq \frac{C_P^2M}{2(\tau-\tau_c)m}\Vert \nabla \Ff'[\mu_t]\Vert^2_{L^2(\mu_t)}.
$$ 
This can be interpreted as a Polyak--\L{}ojasiewicz inequality in the Wasserstein geometry (see~\cite{blanchet2018family}) over the set of probability densities in the range $[m,M]$.
Since $\frac{\d}{\d t} (\Ff(\mu_t)-\Ff(\mu^*))=-\Vert \nabla \Ff'[\mu_t]\Vert^2_{L^2(\mu_t)}$ for almost every $t$, we deduce the exponential convergence rate by Gr\"onwall's inequality.

Let us now consider the critical case $\tau\geq \tau_c$ (the argument that follows is in fact valid for a WGF of any convex function $\Ff$). 
For $t\geq 0$, let $h(t)=\Ff(\mu_t)-\Ff(\mu^*)$. By the convexity, Cauchy-Schwarz in $L^2(\lambda)$, and Poincaré inequalities respectively, it holds 
\begin{multline*}
h(t) \leq \int \Ff'[\mu_t]\d( \mu_t-\mu^*)
\leq \Big\Vert \frac{\d\mu_t}{\d\lambda}-\frac{\d\mu^*}{\d\lambda}\Big\Vert_{L^2(\lambda)}\sqrt{\Var_\lambda(\Ff'[\mu_t])}  \leq (2M\sqrt{|\Omega|})\cdot  C_P \cdot \Vert \nabla \Ff'[\mu_t]\Vert_{L^2(\lambda)}.
\end{multline*}
It follows
\begin{align*}
h'(t)= - \Vert \nabla \Ff'[\mu_t] \Vert^2_{L^2(\mu_t)} \leq -m \Vert \nabla \Ff'[\mu_t]\Vert_{L^2(\lambda)}^2 \leq -\frac{m}{(2M)^2|\Omega|C_P^2} h(t)^2.
\end{align*}
In particular $h$ is nonincreasing and therefore, either $h(t)= 0$ and the claim is trivially true, or $h(s)>0$ for all $0\leq s\leq t$ and in this case we have $(1/h)'(s) \geq \frac{m}{(2M)^2|\Omega|C_P^2}$. Integrating over $[0,t]$ yields $1/h(t)\geq \frac{m}{(2M)^2|\Omega|C_P^2} t+ 1/h(0)$ and the conclusion follows.
\end{proof}

We would like now to obtain the density bounds in the assumptions of Lemma~\ref{lem:compact-convergence}. 

We start with the following proposition, which shows explicit upper and lower bounds for the transition kernel of the Fokker--Planck equation with bounded drift in $\R^d$. These bounds are standard, but we were not able to find a reference tracking the dependency in the parameters of the problem. We will then use them below to obtain explicit density estimates in the torus.

\begin{proposition}
\label{prop:heat_kernel_Rd}
    Let $\Omega = \R^d$. Let \( b: [0,T] \times \Omega \to \Omega \) be a measurable vector field such that
\[
{\bar M} := \sup_{(t,x) \in [0,T] \times \Omega} |b(t,x)| < \infty.
\]
Consider the Fokker--Planck equation
\[
\partial_t \rho(t,x) = - {\rm div} (b(t,x) \rho(t,x))+\tfrac{1}{2} \Delta \rho(t,x) ,\quad\text{for}\quad (t, x)\in [0, T]\times \Omega,
\]
and let  \( k(t,x,y) \) denote the transition kernel (or density) of this equation (i.e., the solution with $\rho(0, x) = \delta_y(x)$). Alternatively, $k(t, x, y)$ is the transition density of $Y_t$, the solution to the SDE
\[
\d Y_t = b(t, Y_t) \, \d t + \d B_t, \quad X_0 = y,
\]
where \( (B_t) \) is a standard Brownian motion in \( \Omega \). Then, we have the bounds
$$
2^{-1/2} e^{-{\bar M}^2(|y-x|+2\sqrt{t})^2-{\bar M}^2t}\leq (2\pi t)^{d/2} e^{\frac{|x-y|^2}{2t}} k(t,x,y)\leq 2^{1/2} e^{{\bar M}^2(|y-x| + 2\sqrt{t})^2},
$$
for all $t > 0$, $x, y\in \Omega$.
\end{proposition}
\begin{proof}

Fix $t>0$ and let $\PP_{x} \in \Pp(\Cc([0,t];\RR^d))$ be the law of $Y$  and  $\WW_{x} \in \Pp(\Cc([0,t];\RR^d))$ be the law of the standard Brownian motion starting at $x$, both restricted to the time interval $[0,t]$. For $X$ distributed according to $\WW_{x}$, consider the stochastic exponential of the martingale $\int_0^s b(s,X_s)^\top \d X_s$, given at time $t$ by
$$
Z_t\coloneqq \exp\left( \int_{0}^{t} b(s,X_s)^\top \d X_s -\frac12 \int_{0}^{t} \vert b(s,X_s)\vert^2 \d s\right).
$$
Novikov's integrability condition is trivially satisfied since $b$ is bounded so by Girsanov's formula, it holds $ \frac{\d \PP_{x}}{\d \WW_{x}}=Z_t$.  Therefore the probability density $k(t,x,\cdot)$ of $Y$ at time $t$ satisfies, for any Borel set $A\subset \RR^d$,
$$
\int_A k(t,x,y)\d y =\E_{\PP_x}(\ones_{\{Y_t\in A\}}) =  \E_{\WW_x}(Z_t\ones_{\{X_t\in A\}}) = \int_A \E_{\WW_x}[Z_t\mid X_t=y] q(t,x,y)\d y
$$ 
where $q(t,x,\cdot )$ is the probability density  at time $t$ of the Brownian motion starting from $x$. Conditioned on $X_t=y$, $X$ is a Brownian bridge connecting $x$ at time $0$ to $y$ at time $t$. It is a weak solution to the stochastic differential equation
\begin{align}\label{eq:bridge}
\d X_s = \frac{y-X_s}{t-s} \d s +\d B_s,\quad X_0=x
\end{align}
where $B$ is again a standard Brownian motion.

Let us start with the upper bound. Using the decomposition~\eqref{eq:bridge} of the Brownian bridge and  Cauchy-Schwarz's inequality, we have
$$
\E_{\WW_x}[Z_t\mid X_t=y] \leq \left(\E e^{2\int_0^t b(s,X_s)^\top \frac{y-X_s}{t-s}\d s}\right)^\frac12\left( \E e^{2\int_0^t b(s,X_s)^\top \d B_s -\int_0^t|b(s,X_s)|^2\d s}\right)^\frac12.
$$
The second factor is the square-root of the expectation of an exponential martingale and thus equals $1$. For the first factor, let us exploit the properties of the subgaussian norm, which is defined for a $\RR$-valued random variable $X$ as 
$$
\Vert X\Vert_{\psi_2} = \inf \{c>0\;;\; \E e^{X^2/c^2}\leq 2\}
$$
and for a $\RR^d$-valued random variable $X$ as $\Vert X\Vert_{\psi_2} = \sup_{|u|\leq 1} \Vert u^\top X\Vert_{\psi_2}.
$
Using the boundedness of $b$ and the triangle inequality, it holds
\begin{align}
\Big\Vert 2\int_0^t b(s,X_s)^\top \frac{y-X_s}{t-s}\d s\Big\Vert_{\psi_2} &\leq 2{\bar M}\int_0^t \frac{\Vert y-X_s\Vert_{\psi_2}}{t-s}\d s
\end{align}
Now, since $X$ here is a Brownian bridge, the law of $X_s$ is Gaussian with mean $x+\frac{s}{t}(y-x)$ and entrywise variance $\sigma^2_s=s(1-s/t)$. Therefore $y-X_s$ is Gaussian with mean $m_s=(1-s/t)(y-x)$ and covariance matrix $\sigma^2_s{\rm Id}$, and we have (using that $\Vert \mathcal{N}(0, 1)\Vert_{\psi_2} = \sqrt{8/3}$), 
\[
\begin{split}
\Vert y-X_s\Vert_{\psi_2} & \le \Vert y-X_s-m_s\Vert_{\psi_2} +\Vert m_s\Vert_{\psi_2}  \\
& \le \sigma_s \sqrt{8/3} + |m_s|/{\log^{1/2}(2)}  = \sqrt{8s(1-s/t)/3}+(1-s/t)|y-x|/\log^{1/2}(2).
\end{split}
\]
It follows
\begin{align}
\Big\Vert 2\int_0^t b(s,X_s)^\top \frac{y-X_s}{t-s}\d s\Big\Vert_{\psi_2} &\leq 2{\bar M}\int_0^t \Big(\frac{|y-x|}{\log^{1/2}(2) \, t}+\sqrt{\frac{8s}{3t(t-s)}}\Big)\d s\\
&\le 2{\bar M} |y-x|/\log^{1/2}(2) + \pi \sqrt{8/3}{\bar M}\sqrt{t},
\end{align}
since $\int_0^t\sqrt{\frac{s}{t-s}}\d s=t\int_0^1\sqrt{\frac{u}{1-u}}\d u=t\pi/2$. For a real random variable $X$, since $X\leq \frac{\Vert X\Vert^2_{\psi_2}}{4}+\frac{X^2}{\Vert X\Vert_{\psi_2}^2}$, it holds $\E[e^{X}]\leq 2e^{\Vert X\Vert_{\psi_2}^2/4}$, so we deduce
$$
\left(\E e^{2\int_0^t b(s,X_s)^\top \frac{y-X_s}{t-s}\d s}\right)^\frac12 \leq 2^{1/2} e^{(2{\bar M} |y-x|/\log^{1/2}(2) + \pi \sqrt{8/3}{\bar M}\sqrt{t})^2/8} \leq 2^{1/2}e^{{\bar M}^2(|y-x| + 2\sqrt{t})^2}.
$$
All in all, we have proved the upper bound
$$
k(t,x,y)= \E_{\WW_x}[Z_t\mid X_t=y]q(t,x,y) \leq 2^{1/2}(2\pi t)^{-d/2} e^{{\bar M}^2(|y-x| + 2\sqrt{t})^2 -|x-y|^2/(2t)}.
$$
For the lower bound, we use that $\E_{\WW_x}[Z_t\mid X_t=y]\geq \big(\E_{\WW_x}[Z_t^{-1}\mid X_t=y]\big)^{-1}$ by Jensen's inequality on the convex function $u\mapsto 1/u$ for $u>0$. Moreover, using again the expression~\eqref{eq:bridge} for the Brownian bridge,
\begin{align*}
\E_{\WW_x}[Z_t^{-1}\mid X_t=y] &= \E e^{-\int_0^t b(s,X_s)^\top \d X_s +\frac12\int_0^t |b(s,X_s)|^2\d s}\\
&= \E e^{-\int_0^t b(s,X_s)^\top(\frac{y-X_s}{t-s}) \d s-\int_0^t b(s,X_s)^\top\d B_s-\frac12\int_0^t |b(s,X_s)|^2\d s +\int_0^t |b(s,X_s)|^2\d s}\\
&\leq e^{{\bar M}^2t}\left(\E e^{-2\int_0^t b(s,X_s)^\top \frac{y-X_s}{t-s}\d s}\right)^\frac12\left( \E e^{-2\int_0^t b(s,X_s)^\top \d B_s -\int_0^t|b(s,X_s)|^2\d s}\right)^\frac12\\
&= e^{{\bar M}^2t}\left(\E e^{-2\int_0^t b(s,X_s)^\top \frac{y-X_s}{t-s}\d s}\right)^\frac12\\
&\leq 2^{1/2} e^{{\bar M}^2 t+ {\bar M}^2(|y-x| + 2\sqrt{t})^2} 
\end{align*}
where we have replaced $b$ by $-b$ in the previous analysis. Combining the previous estimates, we have proved the desired two-sided bounds. 
\end{proof}

Thanks to the previous transition kernel estimates, we can deduce upper and lower density bounds for the transition kernel in the torus by considering its periodic representation in $\R^d$:

\begin{corollary}
    \label{cor:heat_kernel_Td}
    In the context of Proposition~\ref{prop:heat_kernel_Rd}, let us assume $\Omega = \T^d$, and let $\bar k$ denote the corresponding transition kernel or density. Then, we have the bounds 
    $$
5^{-1} 3^{-d} e^{-\frac32 {\bar M}^2 d} \leq  t_*^{ d/2} \bar k(t_*,x,y)\leq 4\cdot 2^d
$$
for $t_* := \frac{1}{\max\{8{\bar M}^2, 1\}}$ and for all $x, y\in \Omega$.
    \end{corollary}
   \begin{proof}
 We bound the kernel $\bar k$ thanks to the bounds derived in Proposition~\ref{prop:heat_kernel_Rd} for $k$ and thanks to the relation between the two kernels
    \begin{align}\label{eq:torus-Rd}
\bar k(t,x,y) = \sum_{z\in \ZZ^d} k(t,x,y+z),
\end{align}
 which follows by
  the periodic extension of $\T^d$ towards $\R^d$.

For the lower bound at time $t_* := \frac{1}{\max\{8{\bar M}^2, 1\}}$, for every $x, y$ we consider only one term in the sum, namely the $z \in \mathbb Z^d$ closest to $x-y$. Since their distance is at most $\sqrt d /2$, we get 
\begin{equation}
    \label{eq:torus_low}
    2^{1/2} (2\pi t_*)^{d/2}\bar k (t, x, y) \ge  e^{-9{\bar M}^2 t_*} \sum_{z\in \Z^d} e^{-\left(2{\bar M}^2+\frac{1}{2t_*}\right)|y-x-z|^2}\ge  e^{-9{\bar M}^2 t_*}  e^{-\left(\frac{{\bar M}^2}{2}+\frac{1}{8t_*}\right)d},
\end{equation}
which implies the lower bound (since $\frac{{\bar M}^2}{2}+\frac{1}{8t_*} \le \frac{3}{2} {\bar M}^2+\frac18$).

For the upper bound, we first observe that, thinking the Poisson sum as a discretization of the integral, for any $\alpha>0$, $x\in \mathbb R$
\begin{equation}
\label{eqn:poissonsum}
\sum_{j\in \Z} e^{-\alpha(j - x)^2} \le 1 + \int_{\mathbb R}e^{-\alpha(j - x)^2}= 1+ \sqrt{\frac\pi\alpha}.
\end{equation}
From Proposition~\ref{prop:heat_kernel_Rd} and the triangle inequality
\begin{equation}
    \label{eq:torus_upp}
      \frac{(2\pi t_*)^{d/2}}{2^{1/2}}\bar k (t_*, x, y) \le e^{8{\bar M}^2 t_*} \sum_{z\in \Z^d} e^{\left(2{\bar M}^2-\frac{1}{2t_*}\right)|y-x-z|^2}\hspace{-2mm}= 
      e^{8{\bar M}^2 t_*} \prod_{i=1}^d\sum_{z_i \in \mathbb Z} e^{\left(2{\bar M}^2-\frac{1}{2t_*}\right)|y_i-x_i-z_i|^2},
\end{equation}
 Since $2{\bar M}^2-\frac{1}{2t_*} = \min\left\{-2{\bar M}^2, 2{\bar M}^2 - \frac12 \right\} \le -\frac{1}{4}$, we apply the right-hand side of \eqref{eqn:poissonsum} with $\alpha= 1/4$ to bound the right-hand side in \eqref{eq:torus_upp} by $
(1+2 \sqrt\pi)^d$.
    \end{proof}

    \begin{lemma}[Density bounds]\label{lem:density-bounds-torus}
 Let $\Omega=\TT^d$ and assume that $\Ff=\Gg+\tau\Hh$ with $\Gg$ satisfying Assumption~\ref{ass:regularity-G}, and let $\mu_t$ be a WGF~\eqref{eq:PDE} starting from $\mu_0\in \Pp_2(\Omega)$. Let us assume, moreover, that $\Vert \nabla \Gg'[\mu_t]\Vert_\infty\leq {\bar M}$ for all $t\geq 0$, for some ${\bar M} \ge \tau$. Then, for  all $t\geq \frac{\tau}{4{\bar M}^2}$, it holds  
     $$
5^{-1} e^{-\frac38 ({\bar M}/\tau)^2 d} ({\bar M}/\tau)^d (\sqrt{2}/3)^d \leq   \mu_t \leq 4\cdot (2\sqrt{2})^d ({\bar M}/\tau)^{d}.
$$
\end{lemma}
\begin{proof}
    Dividing the objective function by $2\tau$, we can without loss of generality focus on the case $\tau=1/2$ (this leads to the same dynamics, and therefore the same uniform-in-time density bounds, up to a linear time rescaling).
In order to reach our conclusion, we now use the  upper and lower bounds 
on the transition kernel $\bar k(t,x,y)$ obtained in Corollary~\ref{cor:heat_kernel_Td}, thanks to the representation
$$
\mu_{t}(y) = \int \bar k(t, x, y) \mu_{0}(x)\d x.
$$
Using that $\mu_0$ (and $\mu_t$) is a probability measure, and translating in time if necessary, we reach the desired bounds.
\end{proof}

We now have all the ingredients to proceed with the proof of Theorem~\ref{thm:compact}:

\begin{proof}[Proof of Theorem~\ref{thm:compact}]
We directly verify that Lemma~\ref{lem:compact-convergence} applies in this setting: on the one hand,  the torus $\TT^d$ satisfies a Poincaré inequality $\Vert f\Vert_{L^2}\leq (2\pi)^{-1} \Vert \nabla f\Vert_{L^2}$ for $f\in H^1(\TT^d)$ with mean $0$; and on the other hand, the  density bounds on $\mu_t$ are given by Lemma~\ref{lem:density-bounds-torus}  for $t\geq \tau/(4L^2)$, with $M= 4\cdot (2\sqrt{2})^d ({L}/\tau)^{d}$ and $m=5^{-1} e^{-\frac38 ({L}/\tau)^2 d} (L/\tau)^d (\sqrt{2}/3)^d$.
\end{proof}

  \section{The noncompact case}\label{sec:mainproof}

In this section, we first state a few lemmas that establish the convergence of WGFs under abstract comparability conditions (which will then be proved in the next section), and then conclude with the proof of Theorem~\ref{thm:main}.

The first lemma is about the strongly convex case:

  \begin{lemma}\label{lem:strongly-cvx-Rd-principle}
  Let $\Omega=\RR^d$. 
   Assume that $\Ff=\Gg+\tau\Hh$ is convex and $\Gg$ satisfies Assumption~\ref{ass:regularity-G}, with $\tau>\tau_c$ (recall that $\tau_c$ is given by~\eqref{eq:tau_c}). Let $\mu_t$ be a WGF~\eqref{eq:PDE} starting from $\mu_0\in \Pp_2(\Omega)$. Assume that there exist two isotropic centered Gaussian measures $\gamma_1=\Nn(0,\sigma_1^2I)$ and $\gamma_2=\Nn(0,\sigma_2^2I)$ with $\sigma_1<\sigma_2$ and $m,M>0$ such that for all $t\in \RR_+$, $m\gamma_1\leq  \mu_t \leq M\gamma_2$. Then for any $0<\kappa<\frac{\sigma_1^2}{\sigma_2^2-\sigma_1^2}$, there exists $C>0$ depending only on $\tau$, $\sigma_1$, $\sigma_2$, $m$, $M$, $\kappa$, and $L$, such that 
  $$
  \Ff(\mu_t)-\inf\Ff \leq \left(  (\Ff(\mu_0)-\inf\Ff)^{-1/\kappa} +C(\tau - \tau_c)^{1+1/\kappa}t\right)^{-\kappa}
  $$
  \end{lemma}
  This lemma shows that in the strongly convex regime $\tau>\tau_c$, sandwiching the density $\mu_t$ between two Gaussian distributions leads to a polynomial rate of convergence, with a rate that gets faster as the covariances of the two Gaussians get closer. Note that if $\mu_t$ satisfies a Poincaré inequality uniformly (or if we had $\sigma_1=\sigma_2$) then an analogous argument would show exponential convergence;  but proving such a property on $\mu_t$ appears out of reach with our approach. We also note that  for $\kappa=1$, the condition on the variances matches that in Lemma~\ref{lem:poly-Rd-principle} below (in this regime, one should rather apply the latter, which has better constants and allows $\tau=\tau_c$).

  In order to show Lemma~\ref{lem:strongly-cvx-Rd-principle} (and later on, the convex case in Lemma~\ref{lem:poly-Rd-principle}), we start with the following result:  
   \begin{lemma}\label{lem:variance_change}
  Let $\gamma_i=\Nn(0,\sigma_i I)$, $i\in \{1,2\}$ be two Gaussian densities in $\RR^d$ and for a measurable $f:\RR^d\to \RR$, let $m_i=\int f\d\gamma_i$. Letting $p> 1$ and $\alpha=\sigma_1^2/\sigma_2^2 > 1/p$, it holds
    \begin{align*}
  \Var_{\gamma_2}(f)\leq C \cdot \Vert f - m_1\Vert^2_{L^{2p}(\gamma_1)} \quad \text{with}\quad 
  C =  \alpha^{d/2}\left(\frac{p-1}{p\alpha-1}\right)^{\frac{d(p-1)}{2p}}.
  \end{align*}
  \end{lemma}
  \begin{proof}
  Let $\rho = \frac{\d\gamma_2}{\d \gamma_1}$. It holds with $\frac1p +\frac1q=1$ by H\"older's inequality
  $$
  \Var_{\gamma_2}(f)\leq \int (f-m_1)^2\rho \d\gamma_1 \leq \left(\int \rho^q\d\gamma_1\right)^{\frac1{q}} \left(\int (f-m_1)^{2p}\d\gamma_1\right)^{\frac1{p}}.
  $$
  Therefore, the claim holds with
$$
C^q =\int \rho^q\d\gamma_1= \frac{1}{(2\pi \sigma_1^2)^{d/2}}\left(\frac{\sigma_1}{\sigma_2}\right)^{dq} \int \exp\left[-\frac{\vert x\vert^2}{2}\left( \Big(\frac{1}{\sigma_2^2}-\frac{1}{\sigma^1_2} \Big)q +\frac{1}{\sigma_1^2}\right) \right]\d x.
$$
This quantity is finite if and only if  $\frac{1}{\tilde \sigma^2} \coloneqq \Big(\frac{1}{\sigma_2^2}-\frac{1}{\sigma^1_2} \Big)q +\frac{1}{\sigma_1^2}>0$,  that can be rewritten as $(\sigma_1/\sigma_2)^2>\frac{q-1}{q}=\frac{1}{p}$, which is the stated condition. When this condition is satisfied, we have
$$
\frac{\sigma_1^2}{\tilde \sigma^2} = \left(\frac{\sigma_1}{\sigma_2}\right)^{2}q -q+1
$$
and therefore
\begin{align*}
C^q=(2\pi \sigma_1^2)^{-d/2} \left(\frac{\sigma_1}{\sigma_2}\right)^{dq} (2\pi \tilde \sigma^2)^{d/2}=\left(\frac{\tilde \sigma}{\sigma_1}\right)^{d}  \left(\frac{\sigma_1}{\sigma_2}\right)^{dq} =\left[\alpha q -q+1\right]^{-d/2}\alpha^{dq/2}
\end{align*}
and the expression for $C$ follows by taking the $q$-th root.
\end{proof}

Combined with Lemma~\ref{lem:PL}, this yields the strongly convex case:
  
  \begin{proof}[Proof of Lemma~\ref{lem:strongly-cvx-Rd-principle}]

As in the proof of Lemma~\ref{lem:compact-convergence}, we may argue that $\mu_t$ is converging to some minimizer $\mu^*$ with the same upper and lower bounds.  By  Lemma~\ref{lem:PL} applied with the reference measure $\lambda=\gamma_2$, we have for all $t\geq 0$
  \begin{align*}
  \Ff(\mu_t)-\Ff(\mu^*) &\leq \frac{M}{2(\tau-\tau_c)} \Var_{\gamma_2}(\Ff'[\mu_t]).
  \end{align*}
  Under our condition on $\kappa$, it holds $1<\frac{\sigma_2^2}{\sigma_1^2}<1+\kappa^{-1}$. We pick $p_0$ to be the midpoint of the interval $ {[\frac{\sigma_2^2}{\sigma_1^2}, 1+\kappa^{-1}]}$ (which has nonempty interior under our assumptions) and apply Lemma~\ref{lem:variance_change} with this choice of $p_0$ to get
  $$
  \Var_{\gamma_2}(\Ff'[\mu_t]) \leq C_1 \Vert \Ff'[\mu_t]-m_1\Vert^2_{L^{2p_0}(\gamma_1)},
  $$
where $m_1 = \int \Ff'[\mu_t]\d\gamma_1$ (and $C_1$ is explicit in the lemma). Then by H\"older's inequality with $p=(1+\kappa^{-1})/p_0$ and $\frac1p+\frac1q=1$ (our choice of $p_0$ ensures $p>1$ and therefore $q<+\infty$) it holds
\begin{align*}
\Vert \Ff'[\mu_t]-m_1\Vert^2_{L^{2p_0}(\gamma_1)} &=\left(\int (\Ff'[\mu_t]-m_1)^{(2/p)+2(p_0-1/p)}\d\gamma_1\right)^{1/p_0}\\
&\leq \left(\int (\Ff'[\mu_t]-m_1)^{2}\d\gamma_1\right)^{1/(pp_0)} \underbrace{\left(\int (\Ff'[\mu_t]-m_1)^{2q(p_0-1/p)}\d\gamma_1\right)^{1/(qp_0)}}_{(I)}
\end{align*}
Remember that $\Ff'[\mu_t]=\Gg'[\mu_t]+\tau \log(\mu_t)$. Under Assumption~\ref{ass:regularity-G} and thanks to the Gaussian bounds on $\mu_t$, there exists $A,B$ (that depend only on $L, \tau$ and $\sigma_1$) such that $| \Ff'[\mu_t](x)|\leq A+B\vert x\vert^2$. Therefore, the factor $(I)$ can be bounded by a quantity $C_2$ that only depends on $L, \tau, \sigma_1$, and the exponents. We deduce that
$$
 \Var_{\gamma_2}(\Ff'[\mu_t]) \leq C_1 \Vert \Ff'[\mu_t]-m_1\Vert^2_{L^{2p_0}(\gamma_1)}\leq C_1C_2 \Var_{\gamma_1}(\Ff'[\mu_t])^{1/(pp_0)}.
$$
(Notice the loss of exponent which is due to the change of Gaussian measure.)
 On the other hand, by the Poincaré inequality for the measure $\gamma_1$, it holds: 
  $$
\Var_{\gamma_1}(\Ff'[\mu_t]) \leq  \sigma_1^2 \Vert \nabla \Ff'[\mu_t]\Vert^2_{L^{2}(\gamma_1)} \leq \frac{\sigma_1^2}{m} \Vert \nabla \Ff'[\mu_t]\Vert^2_{L^{2}(\mu_t)}.
  $$
  Gathering all the inequalities derived so far and letting $h(t)=\Ff(\mu_t)-\Ff(\mu^*)$, we get
  $$
  h(t) \leq \frac{MC_1C_2}{2(\tau-\tau_c)}\left(\frac{\sigma_1^2}{m} \Vert \nabla \Ff'[\mu_t]\Vert^2_{L^{2}(\mu_t)} \right)^{1/(1+\kappa^{-1})}
  $$
 Since $h'(t)=\Vert \nabla \Ff'[\mu_t]\Vert^2_{L^{2}(\mu_t)}$, this can be re-expressed as
 $$
 h(t)^{1+\kappa^{-1}}\leq \left(\frac{MC_1C_2}{2(\tau-\tau_c)}\right)^{1+\kappa^{-1}}\frac{\kappa\sigma_1^2}{m} \Big(-\kappa^{-1}h'(t)\Big) = \frac{1}{C(\tau-\tau_c)^{1+1/\kappa}}\Big(-\kappa^{-1}h'(t)\Big)
 $$
 where the last expression defines $C$. Therefore we have $(h(t)^{-1/\kappa})'\geq C(\tau-\tau_c)^{1+1/\kappa}$ hence $h(t)^{-1/\kappa}\geq h(0)^{-1/\kappa}+tC(\tau-\tau_c)^{1+1/\kappa}$ and finally the conclusion $h(t)\leq (h(0)^{-1/\kappa}+tC(\tau-\tau_c)^{1+1/\kappa})^{-\kappa}$. 
 \end{proof}

In the (simply) convex case, we have instead:

\begin{lemma}\label{lem:poly-Rd-principle} Let $\Omega=\RR^d$ and assume that $\Ff=\Gg+\tau\Hh$ is convex and $\Gg$ satisfies Assumption~\ref{ass:regularity-G}. Let $\mu_t$ be a WGF~\eqref{eq:PDE} starting from $\mu_0\in \Pp_2(\Omega)$
Suppose that there exist  $\sigma_1>0$, $\sigma_2\in {[\sigma_1, \sqrt{2}\sigma_1[}$ and $m,M>0$ such that $m\gamma_1\leq \mu_t\leq M\gamma_2$ for all $t\geq 0$, where $\gamma_i=\Nn(0,\sigma^2_iI)$, $i\in \{1,2\}$. Then 
$$
\Ff(\mu_t)-\inf\Ff\leq \left((\Ff(\mu_0)-\inf \Ff)^{-1}+\frac{m^2}{4CM^2\sigma_1^2} t\right)^{-1}
$$
where $C=(2\alpha-\alpha^2)^{-d/2}$ and $\alpha=\sigma_2^2/\sigma_1^2$ (it holds $\alpha<2$ under our assumptions).
\end{lemma}
\begin{proof}
As a preliminary computation, remark that
\begin{align*}
C\coloneqq \int \left(\frac{\gamma_2}{\gamma_1}\right)^2\d\gamma_1 &= \frac{1}{(2\pi \sigma_1^2)^{d/2}}\left(\frac{\sigma_1}{\sigma_2}\right)^{2d} \int \exp\left(-\frac{\vert x\vert^2}{2}\left[ \Big(\frac{1}{\sigma_2^2}-\frac{1}{\sigma^1_2} \Big)2 +\frac{1}{\sigma_1^2}\right] \right)\d x
\end{align*}
and this quantity is finite if and only if the bracketed quantity in the exponential is positive, which is equivalent to $\sigma_2<\sqrt{2}\sigma_1$. A direct computation as in Lemma~\ref{lem:variance_change} shows that $C=(2\alpha-\alpha^2)^{-d/2}$.

 Again,  as in the proof of Lemma~\ref{lem:compact-convergence}, we may argue that $\mu_t$ is converging to some minimizer $\mu^*$ with the same upper and lower bounds. Now for $t\geq 0$, let $c_t=\int \Ff'[\mu_t]\d\gamma_1$ and let $h(t)=\Ff(\mu_t)-\Ff(\mu^*)$. By the convexity inequality for $\Ff$ and the Cauchy-Schwartz inequality in $L^2(\gamma_1)$, it holds
\begin{align*}
h(t) &\leq \int \Ff'[\mu_t]\d (\mu_t-\mu^*)= \int (\Ff'[\mu_t]-c_t)\frac{\mu_t-\mu^*}{\gamma_1}\d \gamma_1 \leq \Vert \Ff'[\mu_t]-c_t\Vert_{L^2(\gamma_1)} \Big\Vert \frac{\mu_t}{\gamma_1}-\frac{\mu^*}{\gamma_1}\Big\Vert_{L^2(\gamma_1)}.
\end{align*}
The second factor can be bounded as
\begin{align*}
\Big\Vert \frac{\mu_t}{\gamma_1}-\frac{\mu^*}{\gamma_1}\Big\Vert_{L^2(\gamma_1)}\leq \Big\Vert \frac{\mu_t}{\gamma_1}\Big\Vert_{L^2(\gamma_1)}+ \Big\Vert \frac{\mu^*}{\gamma_1}\Big\Vert_{L^2(\gamma_1)}\leq 2M\sqrt{C}.
\end{align*}
Now since $\gamma_1$ satisfies the Poincaré inequality $\Var_{\gamma_1}(f)\leq \sigma_1^2 \int \vert \nabla f\vert^2\d\gamma_1$,  it follows
$$
h(t)\leq {2\sigma_1M\sqrt{C}} \Vert \nabla \Ff'[\mu_t]\Vert_{L^2(\gamma_1)} \leq  \frac{2\sigma_1M\sqrt{C}}{m }\Vert \nabla \Ff'[\mu_t]\Vert_{L^2(\mu_t)}.
$$
Therefore, it holds for almost every $t\geq 0$,
$$
h'(t) = \Vert \nabla \Ff'[\mu_t]\Vert^2_{L^2(\mu_t)}\leq - \frac{m^2}{4M^2\sigma_1^2C}h(t)^2
$$
In particular $h$ is nonincreasing and therefore, either $h(t)= 0$ and the claim is trivially true, or $h(s)>0$ for all $0\leq s\leq t$ and in this case we have $(1/h)'(s) \geq \frac{m^2}{4M^2\sigma_1^2C}$. Integrating over $[0,t]$ yields $1/h(t)\geq \frac{m^2}{4M^2\sigma_1^2C} t+ 1/h(0)$ and the conclusion follows.
\end{proof}

Now, a combination of Lemma~\ref{lem:strongly-cvx-Rd-principle} (in the strongly convex case) or  Lemma~\ref{lem:poly-Rd-principle} (in the convex case), with the Gaussian bounds obtained (independently of this section) in Proposition~\ref{prop:bounds} below, gives the main result, Theorem~\ref{thm:main}: 

{\begin{proof}[Proof of Theorem~\ref{thm:main}] Under the assumptions of the theorem, a WGF $(\mu_t)$ satisfies an equation of the form
$$\partial_t \mu_t = {\rm div}((-\alpha x +\bv_t(x))\mu_t)+\tau \Delta \mu_t$$
where $\bv_t(x)=\nabla V[\mu_t](x)$ satisfies $\Vert \bv_t \Vert_\infty \leq L$. Then, we have that $(\tilde \mu_t)_{t\geq 0}$ defined by $\tilde \mu_t(x)=b^d \mu_{at}(bx)$ with $a=1/\alpha$ and $b=\sqrt{2\tau/\alpha}$ is a path in $\Pp_2(\RR^d)$ that solves 
$$
\partial_t\tilde \mu_t ={\rm div}((-x+\tilde \bv_t(x)) \tilde \mu_t) + \frac12 \Delta \tilde \mu_t.
$$
with $\tilde \bv_t(x) = (a/b) \bv_{at}(bx)$. In particular, with this transformation we have $\Vert \tilde \bv_t\Vert_\infty\leq aL/b = L/\sqrt{2\alpha \tau}$, and the assumption $\int e^{\vert x\vert^2/M_*^2}\tilde \mu_0(x)\d x<2$ in Proposition~\ref{prop:bounds} is equivalent to 
$$
\int e^{\vert x\vert^2/M_*^2} \mu_0(bx)b^d\d x<2 \iff \int e^{\vert x\vert^2/(b^2M_*^2)}\mu_0(y)\d y<2.
$$
By Proposition~\ref{prop:bounds} with the choice $M_*= M_0/b = M_0\sqrt{\alpha/(2\tau)}$ and  $M=\max\{aL/b,M_*/2\}=\max\{L/\sqrt{2\alpha\tau},M_0\sqrt{\alpha/(8\tau)}\}$, it follows that for all $\eps>0$, there exists $c_\eps>0$ such that for all $t\geq a\cdot T_\eps=\alpha^{-1} (\log(40M_*^2/\eps)+\eps/16)$, it holds
 $$
 c_\eps \left(\frac{\alpha}{2\tau}\right)^{d/2} e^{-(1+\eps)\alpha |x|^2/(2\tau)}\leq 
 \mu_t(x)\leq 
  c_\eps^{-1} \left(\frac{\alpha}{2\tau}\right)^{d/2} e^{-(1-\eps)\alpha |x|^2/(2\tau)}.
 $$
 Let us now consider the two claims separately.

\textbf{Convex case:} To obtain  the $O(1/t)$ rate in the convex case ($\tau\geq \tau_c$): take $\eps=1/4$ in Proposition~\ref{prop:bounds} and let $\sigma_1^2=\frac{\tau}{\alpha(1+\eps)}$ and $\sigma_2^2=\frac{\tau}{\alpha(1-\eps)}$. It holds $\sigma_1^2/\sigma_2^2=3/5>1/2$ so these form valid choices for Lemma~\ref{lem:poly-Rd-principle} for $t \geq\alpha^{-1}(5+\log(M_0^2\alpha/\tau)) \geq \alpha^{-1}(\log(40M_*^2/\eps)+\eps/16)$, and the claim follows by an application of Lemma~\ref{lem:poly-Rd-principle}.

\textbf{Strongly convex case:} Let $\kappa>1$. In order to apply Lemma~\ref{lem:strongly-cvx-Rd-principle}, we need a two-sided Gaussian bound on $\mu_t$ with variances satisfying $\kappa< \sigma_1^2/(\sigma_2^2-\sigma_1^2)$. With the choices of variances as above, we have
\begin{align}
\frac{\sigma_1^2}{\sigma_2^2-\sigma_1^2} = \frac{1}{1+\eps}\left(\frac{1}{1-\eps}-\frac{1}{1+\eps}\right)^{-1} = \frac{1}{2\eps}-\frac12.
\end{align}
Therefore we need to choose $\eps$ so that $\kappa<  \frac{1}{2\eps}-\frac12$, for instance $\eps=(2\kappa+2)^{-1}$ works. Then by Proposition~\ref{prop:bounds}, for any $t\geq \alpha^{-1}(5+\log(M_0^2\alpha\kappa/\tau))>  \alpha^{-1}(\log(40M_*^2/\eps)+\eps/16)$, we have the desired two-sided Gaussian bounds and the result follows by an application of Lemma~\ref{lem:strongly-cvx-Rd-principle}.
\end{proof}

}

\section{Two-sided Gaussian estimates}\label{sec:gaussian}

We consider the advection-diffusion equation:  
\begin{equation}
\label{eq:AD_0}
\partial_t\mu_t -{\rm div}((x-\bv_t(x)) \mu_t) = \tfrac12 \Delta \mu_t,\quad\text{for}\quad (t, x) \in (0, \infty)\times\R^d,
\end{equation}
where $\bv_t = \bv_t(x) :(0, \infty)\times \R^d \to \R^d$ is a given vector field. 

We want to prove the following: 

\begin{proposition}
\label{prop:bounds}
Let $M_*\ge 2$, $M > M_*/2$. For any $\eps \in (0, 1/4)$ there exist $T_\eps > 0$ (which can be taken equal to $\log(40M_*^2/\eps)+\eps/16$) and $c_\eps > 0$ depending only $\eps$, $d$,  $M$, and $M_*$, such that the following statement holds. 

Let $\mu_0\in \mathcal{P}(\R^d)$  be $M_*$-subgaussian, i.e.,  $\int e^{ \vert x\vert^2/M^2_*}\d \mu_0(x)\leq 2$, and let $\bv_t\in L^\infty((0, \infty)\times \R^d; \R^d)$ with 
\begin{equation}
\label{eq:vtbound}
\|\bv_t\|_{L^\infty((0, \infty)\times \R^d)}\le  M.
\end{equation}
Let $\mu_t$ be a solution to \eqref{eq:AD_0} initialized with $\mu_0$. Then, 
\[
0 < c_\eps e^{-(1+\eps) {|x|^2} }\le \mu_t(x) \le c_\eps^{-1} e^{-(1-\eps) |x|^2 }\quad\text{for all}\quad x\in \R^d, \ t \ge T_\eps. 
\]
\end{proposition}

In the next lemma, we first prove a weaker statement, corresponding to obtaining the desired bounds in average on certain balls, rather than pointwise. This lemma follows from a comparison estimate between
 the SDE corresponding to \eqref{eq:AD_0}, namely 
\begin{equation}
\label{eq:SDE}
\d X^{x_0}_t = (\bv_t( X^{x_0}_t) -  X^{x_0}_t)\, \d t + \d B_t,\qquad X^{x_0}_0 = x_0 \in \R^d,
\end{equation}
and the Ornstein-Uhlenbeck stochastic flow, corresponding to the case $\bv_t\equiv 0$,
which can be defined as well as  $\bO^{x_0}_t := \int_0^t e^{s-t}\, \d B_s+e^{-t} x_0$, and whose law of is well-known (see, e.g., \cite[Example 9.17]{daprato2007introstoch}) 
\begin{equation}
\label{eqn:OUtranskern}
p_t^{\rm OU}(x_0, y) =  (\pi(1-e^{-2t}))^{-\frac{d}{2}} \exp\left( {-\frac{|y-e^{-t}x_0|^2}{1-e^{-2t}}}\right)\qquad\text{for}\qquad y\in \R^d. 
\end{equation}

\begin{lemma}
\label{lem:prob}
Let $M_*$, $M$, $\eps$,  $\bv_t$, $\mu_0$, $\mu_t$ as in Proposition~\ref{prop:bounds}. Then,  for any $z\in \R^d$,
\[
\begin{split}
\mu_{t}(B_{M+1}(z)) & \ge c_\eps e^{-(1+\eps)|z|^2} \qquad \,\,\mbox{for any }t\ge \tfrac12\log(1+2/\eps),
\\
\mu_{t}(B_{\eps|z|}(z)) & \le C_\eps   e^{-(1-\eps)|z|^2} \qquad \mbox{for any }t\ge T_\eps = \log(20M_*^2/\eps),
\end{split}
\]
for some $c_\eps$ and $C_\eps$ depending only on $\eps$, $d$, $M$, and $M_0$. 
\end{lemma}
\begin{proof}

A direct computation yields that the solution to the previous SDE is given by 
\[
 X^{x_0}_t = \int_0^t e^{s-t}\bv_s( X^{x_0}_s)\, \d s + \int_0^t e^{s-t}\, \d B_s+e^{-t} x_0.
\]
In particular, denoting $\bO^{x_0}_t := \int_0^t e^{s-t}\, \d B_s+e^{-t} x_0$, we obtain (using the bound on $\bv_s$):
\begin{equation}
    \label{eq:XtOt}
| X^{x_0}_t - \bO_t^{x_0}|\le \int_0^t e^{s-t}|\bv_s( X^{x_0}_s)|\, \d s \le   M (1-e^{-t})\le M.
\end{equation}
We use this information to bound the quantities in the left-hand side of the statement in terms of the Ornstein-Uhlenbeck transition kernel. Since the relation between $\mu_t$ and $X^x_t$ is given by $\mu_t= \mathbb E [(X^\cdot_t)_\# \mu_0]$, we have (also from \eqref{eq:XtOt})
\begin{equation*}
\begin{split}
\mu_{t}(B_{M+1}(z)) &= \int_{\mathbb R ^d} \mathbb P (X_t^x \in B_{M+1}(z)) \, \d \mu_0(x)
\\&\geq  \int_{\mathbb R ^d} \mathbb P (\bO_t^x \in B_{1}(z)) \, \d \mu_0(x)
\\& \geq \mu_0(B_{2M}) \inf_{x\in B_{2M}} \int_{B_{1}(z)} p_t^{\rm OU}(x, y) \, \d y.
\end{split}
\end{equation*}
Also, from $e\int_{\R^d \setminus B_{2M}} \d\mu_0( x) \le \int_{\R^d} e^{\frac{|x|^2}{4M^2}}\d\mu_0( x) \le \int_{\R^d} e^{\frac{|x|^2}{M_*^2}}\d\mu_0( x)\le 2$ (since $M\ge M_*/2$), we deduce $\int_{B_{2M}}\d\mu_0( x) \ge 1-\tfrac{2}{e}> 0$.

To lower bound the second factor, we derive bounds from the explicit expression \eqref{eqn:OUtranskern} of the transition kernel $ p_t^{\rm OU}$ of the Ornstein-Uhlenbeck process. For all $t \geq \tfrac12\log(1+1/\eps)$, $y \in B_{1}(z)$ and $x\in B_{2M}$, using also that from Young inequality $|a+b|^2 \le (1+\delta)|a|^2+(1+\tfrac{1}{\delta})|b|^2$ {for} $a, b\in \R^d$, we have for some $c$ dimensional
\begin{equation*}
\begin{split}  p_t^{\rm OU}(x, y) &\ge c\exp\left( {- (1+\eps) |y-e^{-t}x|^2 }\right) 
\\& \ge c \exp(-(1+2\eps)|z|^2-(1+2/\eps) |y-z+e^{-t}x|^2) 
\\& \ge c \exp(-(1+2\eps)|z|^2-(1+2/\eps) (2M+1)^2).
\end{split}
\end{equation*}
For the upper bound, 
\begin{equation*}
\begin{split}
\mu_{t}(B_{{\eps|z|}}(z)) &= \int_{\mathbb R ^d} \mathbb P (X_t^x \in B_{\eps|z|}(z)) \, \d \mu_0(x)
\\&\leq  \int_{\mathbb R ^d} \mathbb P (\bO_t^x \in B_{M+\eps|z|}(z)) \, \d \mu_0(x)
\\& \leq  \int_{\mathbb R ^d} \int_{ B_{M+\eps|z|}(z)}p^{OU}(x,y) \,\d y \, \d \mu_0(x).
\end{split}
\end{equation*}
We estimate the inner integrand thanks to the explicit expression \eqref{eqn:OUtranskern} of the transition kernel $ p_t^{\rm OU}$, where we denote $a_+ = \max\{a, 0\}$ the positive part,
\begin{equation*}
\begin{split}
\int_{ B_{M+\eps|z|}(z)}p^{OU}(x,y) \,\d y
& \le C (M+\eps|z|)^d \exp\left(-(|z-e^{-t}x|-M-\eps|z|)_+^2\right)
\\& \le C (M+\eps|z|)^d \exp\left(-((1-2\eps)|z|-e^{-t}|x|-M)_+^2\right)
\\& \le C (M+\eps|z|)^d \exp\left(-(1-\eps)(1-2\eps)^2|z|^2+4\eps^{-1} e^{-2t}|x|^2+ 2 \eps^{-1} M^2\right)
\\& \le C_{\eps, M} \exp\left(-(1-2\eps)(1-2\eps)^2|z|^2+4\eps^{-1} e^{-2t}|x|^2\right)
\end{split}
\end{equation*}
Integrating with respect to $\mu_0$, for $t_*\ge 1$, 
\[
\mu_{t_*}(B_{\eps|z|}(z)) \le C   e^{-(1-10\eps)|z|^2}\int_{\R^d} \exp\left(4\eps^{-1} e^{-2t_*} |x|^2\right)\d\mu_0( x).
\]
Since $\mu_0$ is $M_*$-subgaussian, for $T_\eps$ such that $4\eps^{-1} e^{-2T_\eps}\le \frac{1}{M_*^2}$ and $t_*\ge T_\eps$ we get
\[
\mu_{t_*}(B_{\eps|z|}(z)) \le C   e^{-(1-10\eps)|z|^2}.
\]
That is, we can take $T_\eps = \log(2M_*^2/\eps)$, and up to redefining $\eps$, we get the desired result.
\end{proof}

In order to upgrade the integral information in Lemma \eqref{lem:prob} to a pointwise bound on $\mu_t$, we employ transition kernel estimates for our Fokker-Planck equation \eqref{eq:AD_0}, deriving them via self-similar scaling from those involving only a bounded drift term; alternatively, double-sided bounds for the density of the transition kernel $p_t(x, y)$ of a stochastic process $X_t$ satisfying \eqref{eq:SDE}, that is, $\PP(X_t\in A : X_0 = x) = \int_A p_t(x, y)\, \d y$. This type of estimates are, again, known for unexplicit constants (see, e.g., \cite{menozzi2021kernel}); we take the opportunity here to show how to obtain them with explicit computable constants. 

\begin{proposition} \label{prop:two_sided_global}
Let $M > 0$, and consider the Fokker--Planck equation
\begin{equation}
    \label{eq:mu_t_eq}
    \partial_t\mu_t = {\rm div}_x((x-\bv_t(x)) \mu_t) + \tfrac12 \Delta_x \mu_t,\quad\text{for}\quad (t, x) \in (0, \infty)\times\R^d,
\end{equation}
where $\| \bv_t\|_{L^\infty((0, \infty)\times \R^d)}\le M$. Alternatively, let $Y_t$ satisfy 
\begin{equation}
\label{eq:YSDE2}
\d Y_t = ( \bv_t(Y_t) - Y_t)\, \d t + \d B_t,\qquad Y_0 = x.
\end{equation}
Let $p^Y_t(x, y)$ denote the transition kernel for \eqref{eq:mu_t_eq} starting from $\delta_x(\d y)$ (or the transition kernel for $ Y_t$). Then, we have 
\[
c_0 t^{-\frac{d}{2}}\exp\left(-\frac{1}{\lambda_0} \frac{|e^{-t}x - y|^2}{t}\right) \le p^Y_t(x, y) \le c_0^{-1}t^{-\frac{d}{2}}\exp\left(- {\lambda_0} \frac{|e^{-t}x - y|^2}{t}\right)\quad\text{for}\quad t\in (0, 1), 
\]
for some explicit constants $\lambda_0, c_0\in (0, 1)$ depending only on $M$ and $d$. 
\end{proposition}
\begin{proof}
If $\mu_t$ solves \eqref{eq:mu_t_eq} then, for $y = e^t x$ and $s = \frac{e^{2t}-1}{2}$ (or $e^{-t} = (1+2s)^{-1/2}$), we have that 
\[
\nu_s(y) := e^{-d t} \mu_t(x) \quad\Leftrightarrow \quad \mu_t(x) = e^{dt}\nu_{(e^{2t}-1)/2} (e^t x)
\]
satisfies 
\begin{equation}
    \label{eq:nu_s_eq}
\partial_s \nu_s = -{\rm div}_y (\tilde \bv_s(y)\, \nu_s) + \tfrac12 \Delta_y\nu_s,
\end{equation}
where 
\[
\tilde \bv_s(y) = e^{-t}\bv_t(x) = (1+2s)^{-1/2} \bv_{\frac12\log(1+2s)}((1+2s)^{-1/2} y).
\]

In particular, if $p(t, x_1, x_2) = p_t(x_1, x_2)$ is the transition kernel (starting at $t = 0$ with $p(0, \cdot, x_2) = \delta_{x_1}(\d x_2)$) for \eqref{eq:mu_t_eq}, then the transition kernel $q(s, y_1, y_2)$ for \eqref{eq:nu_s_eq} (also starting at $s = 0$, with $q(0, \cdot, y_2) = \delta_{y_1}(\d y_2)$) is given by 
\[
q(s, y_1, y_2) = e^{-dt} p(t,  y_1, e^{-t} y_2) = (1+2s)^{-d/2}p(\tfrac12 \log(1+2s),  y_1, (1+2s)^{-1/2} y_2),
\]
or alternatively, 
\[
p(t, x_1, x_2) = e^{dt} q(\tfrac12 (e^{2t}-1), x_1, e^t x_2). 
\]

The transition kernel $q$ in $\R^d$ satisfies the bounds, given by Proposition~\ref{prop:heat_kernel_Rd}, 
$$
2^{-1/2} e^{-{M}^2(|y_2-y_1|+2\sqrt{s})^2-{M}^2s}\leq (2\pi s)^{d/2} e^{\frac{|y_1-y_2|^2}{2s}} q(s,y_1,y_2)\leq 2^{1/2} e^{{M}^2(|y_2-y_1| + 2\sqrt{s})^2},
$$
which means that $p$ satisfies the bounds
$$
2^{-1/2} e^{-{M}^2(|e^t x_2- x_1|+2\sqrt{s})^2-{M}^2s}\leq (2\pi s)^{d/2} e^{\frac{|x_1-e^t x_2|^2}{2s}} e^{dt} p(t,x_1,x_2)\leq 2^{1/2} e^{{M}^2(|e^t x_2-x_1| + 2\sqrt{s})^2},
$$

That is, the lower bound can be rewritten as (for $t\le 1$, also using that in this interval, $s\in [0, \tfrac12(e^2-1)]\subset [0, 4)$, and $t \le s \le 4t$)
$$
p(t,x_1,x_2) \ge 2^{-1/2} (8\pi t)^{-d/2} e^{-d-36M^2}  e^{-2{M}^2e^{2}|x_2- e^{-t}x_1|^2} e^{-\frac{|e^{-t}x_1- x_2|^2}{1-e^{-2t}}},
$$
which is directly of the desired form, for $\lambda_0$ and $c_0$ small enough depending on $M$ and $d$ (and explicit). 

For the bound above we have, instead, 
$$
  p(t,x_1,x_2)\leq 2^{1/2} e^{32 M^2} (2\pi t)^{-d/2} e^{-\frac{|e^{-t}x_1- x_2|^2}{1-e^{-2t}}} e^{2{M}^2e^{2}| x_2-e^{-t}x_1|^2},
$$
Since $1-e^{-2t}\le 2t$, we have that the result holds for $t$ such that $\frac{1}{2t} \ge 4e^2 M^2$, and in particular, it holds for $t \le \frac{1}{60 M^2}$, for some $\lambda_0$ that depends on $M$ and $c_0$ that, for the upper bound, depends on $d$ only in this time interval. To obtain the result for the whole $t\in (0, 1)$, we use now the semigroup property repeatedly by dividing the interval $(0, 1)$ into $60M^2$ intervals of length $\frac{1}{60 M^2}$, to obtain the desired (explicit) result where now $c_0$ will also depend on $M$. 
\end{proof}

Lemma~\ref{lem:prob} provides us with lower and upper bounds on the transition kernel when integrated on certain balls. Proposition~\ref{prop:two_sided_global} allows us to improve this localized integral information to a pointwise one, by letting a short time elapse.
Now, thanks to the previous results we can proceed to the proof of Proposition~\ref{prop:bounds}:
\begin{proof}[Proof of Proposition~\ref{prop:bounds}]
We divide the proof into two steps, one for the lower bound and one for the upper bound. We fix $z\in \R^d$ and denote  $t_* = T_\eps = \log(20M_*^2/\eps)$ (from Lemma~\ref{lem:prob}).  
\\[0.2cm]
\noindent {\bf Step 1:} 
Let $Y^x_t := X^x_{t_*+t}$, which satisfies \eqref{eq:YSDE2} up to renaming  $ \bv_{t_*+t}$ as $\bv_{t}$, and let $p_t^Y$ be its transition kernel. We then have by Lemma~\ref{lem:prob}
\begin{equation}
\label{eqn:pointwise-w-lower-bound-pt}
\begin{split}
\mu_{t_*+t}(z) & = \int_{\R^d} p_t^Y(x, z) \d\mu_{t_*}(x)\\
& \ge \mu_{t_*}(B_{M+1}(e^{t}z)) \inf_{x\in B_{M+1}(e^t z)} p_t^Y(x, z)\\
& \ge c  e^{-(1+\eps)e^{2t}|z|^2} \inf_{|x-e^t z|\le M+1} p_t^Y(x, z). 
\end{split}
\end{equation}
Taking $t = \frac{\eps}{8}$, the last factor above is lower bounded by a constant depending only on $\eps, M, d$, but otherwise independent of $z$, thanks to Proposition~\ref{prop:two_sided_global} 
\[
\begin{split}
p_t^Y(x, z)&  \ge c_0 t^{-\frac{d}{2}}\exp\left(-\frac{e^{-2t}}{\lambda_0} \frac{| x -e^t z|^2}{t}\right) \ge c_0 t^{-\frac{d}{2}}\exp\left(-\frac{2e^{-2t}}{\lambda_0} \frac{(M+1)^2}{t}\right),
\end{split}
\] 
and hence the right-hand side in \eqref{eqn:pointwise-w-lower-bound-pt} is lower bounded by $c_\eps e^{-(1+2\eps) {|z|^2} }$.

\noindent {\bf Step 2:} By Lemma~\ref{lem:prob} we have $\mu_{t_*}(B_{\eps|z|}(z)) \le C_\eps   e^{-(1-\eps)|z|^2}$.
We now obtain, proceeding as in \eqref{eqn:pointwise-w-lower-bound-pt},  for $t = \eps^2\lambda_0/2$
\begin{equation}
\label{eq:mutstart}
\begin{split}
\mu_{t_*+t}(z)&= \int_{B_{\eps e^t|z|}(e^tz)} p_t^Y(x, z) \d\mu_{t_*}(x) +  \int_{\mathbb R^d \setminus B_{\eps e^t|z|}(e^tz)} p_t^Y(x, z) \d\mu_{t_*}(x) \\
&\le C  e^{-(1-\eps)e^{2t}|z|^2} \sup_{x\in \R^d} p_t^Y(x, z) + \sup_{|x-e^t z|\ge {\eps e^t|z|}} p_t^Y(x, z).
\end{split}
\end{equation}

We know that $p_t^Y(x, z) \le c_0^{-1}t^{-\frac{d}{2}}$, so the first term is bounded by $C_\eps  e^{-(1-2\eps)|z|^2}$ . For the second term we observe that for every $|x-e^t z|\ge\eps e^t|z|$
\[
\begin{split}
p_t^Y(x, z) & \le c_0^{-1}t^{-\frac{d}{2}} \exp\left(-\lambda_0 e^{-2t} \frac{|x-e^t z|^2}{t}\right)  \le c_0^{-1}t^{-\frac{d}{2}}  \exp\left(\frac{\lambda_0}{2}\frac{-\eps^2|z|^2}{t}\right)\le  C_\eps e^{-|z|^2}.
\end{split}
\]
This completes the proof.
\end{proof}

\section{Applications to optimization over the space of measures}\label{sec:applications}
In this section, we discuss a few situations in the context of optimization over the space of probability measures where our analysis leads to new convergence guarantees. To make the discussion more straightforward, we focus on the compact case where $\Omega=\TT^d$ (although in some cases, the results here presented have their analogues in the whole space as well, see e.g. Remark~\ref{rem:noncomp_5}).

We recall that a convex functional satisfying Assumption~\ref{ass:regularity-G} necessarily admits a minimizer (see Remark~\ref{rem:assB}), and we will repeatedly use this existence without justification in what follows. 

\subsection{Entropic regularization of nonconvex objectives}

\paragraph{Convexification of general objectives} In the compact setting, it is convenient to consider, in place of Assumption~\ref{ass:regularity-G}, the following stronger version, which considers regularity of the first-variation with respect to the $L^1$-Wasserstein metric $W_1$.

\begin{assumptionprime}\label{ass:regularity-G-W1} 
Let $\Omega=\TT^d$. For all $\mu\in \Pp(\Omega)$, $\Gg$ admits a first-variation at $\mu$, belonging to $\Cc^1(\Omega;\RR)$. Moreover, there exists $L>0$ such that $\forall \mu_1,\mu_2\in \Pp(\Omega)$ and $x_1,x_2\in \Omega$ it holds
\begin{align}\label{eq:regularity-G-W1}
\vert \nabla \Gg'[\mu_1](x_1) - \nabla \Gg'[\mu_2](x_2)\vert \leq L(W_1(\mu_1,\mu_2)+\dist(x_1,x_2)).
\end{align}
\end{assumptionprime}

Proposition~\ref{prop:reg-to-convex} below motivates why this  stronger regularity assumption is natural in a compact setting: it guarantees that the critical diffusivity $\tau_c$ (defined in~\eqref{eq:tau_c}) is always finite.

\begin{proposition}\label{prop:reg-to-convex}
Let $\Omega=\TT^d$ and let $\Gg$ satisfy Assumption~\ref{ass:regularity-G-W1}. Then, it satisfies Assumption~\ref{ass:regularity-G} and $\Gg+L{\rm diam}(\Omega)^2\Hh$ is convex, that is $\tau_c\leq L{\rm diam}(\Omega)^2$ (with $\tau_c$ given by~\eqref{eq:tau_c}).
\end{proposition}
\begin{proof}
Let us denote $D = {\rm diam}(\Omega)$. From the inequality $W_1(\rho_1,\rho_2)\leq D\Vert \rho_1-\rho_2\Vert_{\rm TV}$ (\cite[Particular case~6.16]{villani2008optimal}), by taking $x_2$ a point of maximum (or minimum) for $\Gg'[\mu_2]$, we have that Assumption~\ref{ass:regularity-G}-\ref{it:Ai} is satisfied (with right-hand side $3DL$). The well-posedness and energy decay of WGFs of functions of the form~\eqref{eq:objective} is now also well-understood: under Assumption~\ref{ass:regularity-G-W1} (in fact, even replacing $W_1$ with $W_2$), one can prove that the objective is semiconvex along generalized $W_2$ geodesics (e.g.~\cite[Lemma~A.2]{chizat2022mean}), and then apply the general theory of WGF from~\cite{ambrosio2005gradient}, or see~\cite[Theorem~3.3]{hu2021mean} for a probabilistic approach. 

 Let now $\rho,\rho'\in \Pp(\Omega)$ and $\rho_t=(1-t)\rho+t\rho'$. It holds
 \begin{align}\label{eq:towards-smooth}
 \Gg(\rho')
 & = \Gg(\rho)+ \int_\Omega \Gg'[\rho](\rho'-\rho)\d x + \int_0^1\int_\Omega (\Gg'[\rho_t]-\Gg'(\rho))(\rho'-\rho)\d x\d t
 \end{align}
Since $\Vert\nabla \Gg'[\rho_t]-\nabla \Gg'[\rho]\Vert_{\infty}\leq LW_1(\rho_t,\rho)$, it follows  writing $D={\rm diam}(\Omega)$: 
\begin{align*}
\Big| \Gg(\rho')-\Gg(\rho)- \int_\Omega \Gg'[\rho](\rho'-\rho)\d x\Big| & \leq LW_1(\rho',\rho) \int_0^1W_1(\rho_t,\rho)\d t
\leq \frac{LD^2}{2} \Vert \rho-\rho'\Vert_{\rm TV}^2
\end{align*}
where we have used again the inequality $W_1(\rho_1,\rho_2)\leq D\Vert \rho_1-\rho_2\Vert_{\rm TV}$. 

This also proves that $\Gg$ is $(-LD^2)$-semiconvex relative to the total variation norm $\Vert \cdot \Vert_{\rm TV}$. On the other hand, by Pinsker's inequality, $\Hh$ is $1$-strongly convex relative to $\Vert \cdot \Vert_{\rm TV}$ over $\Pp(\Omega)$. Therefore, $\Gg+ LD^2\Hh$ is convex, which proves the claim.
 \end{proof}
 
\paragraph{The case of pairwise interactions} Let us discuss the case of functions involving a pairwise interaction energy with a translation invariant interaction kernel:
\begin{align}\label{eq:interaction}
\Gg(\mu) = \Ww(\mu)+\Gg_0(\mu),&& \Ww(\mu) = \int W(y-x)\d\mu(x)\d\mu(y)
\end{align}
where $W\in \Cc^{0}(\TT^d)$ is assumed to be even (coordinate-wise), and $\Gg_0$ is any \emph{convex} function satisfying Assumption~\ref{ass:regularity-G}. In this case, we can obtain finer estimates on $\tau_c$ by doing a spectral decomposition of $\Ww$. Following e.g.~\cite{carrillo2020long}, let us write $W$ as the sum of its Fourier series
 \begin{align*}
 W(z)=\sum_{k\in \ZZ^d} \hat W_k e^{2i\pi k^\top z}, && \hat W_k =\frac{1}{(2\pi)^d} \int W(z) e^{-2i\pi k^\top z}\d z.
 \end{align*}
 Since $W$ is even, the Fourier coefficients $(\hat W_k)_{k\in \ZZ}$ are real numbers and we can decompose $W$ as a difference of convex functions: 
\begin{align}\label{eq:decomposition-kernel}
 \Ww(\mu) = \hat W_0+ \Ww_+(\mu) - \Ww_-(\mu), &&\text{with}&&
 W_\pm(z)= \sum_{k\in \ZZ^d\setminus {0}} (\hat W_k)_\pm e^{2i\pi k^\top z}
 \end{align}
 where $\Ww_\pm(\mu)=\int W_\pm(y-x)\d\mu(x)\d\mu(y)$,  $(a)_+=\max\{a,0\}$, and $(a)_-=\max\{-a,0\}$. The term in $\Ww(\mu)$ associated to the Fourier coefficient $\hat W_0$ is $\int \hat W_0\d\mu(x)\d\mu(y)=\hat W_0$, so it gives 
a constant term, which is trivially convex irrespective of its sign. Also, the convexity of $\Ww_+$ and $\Ww_-$ is clear from the expression
 $
 \Ww_\pm(\mu)= \sum_{k\in \ZZ^d\setminus \{0\}} (\hat W_k)_\pm \Big\vert \int e^{2i\pi k^\top x}\d\mu(x)\Big\vert^2.
 $ These considerations show that the level of entropic regularization that is needed to make the problem convex only depend on the magnitude of the negative spectrum of $W$. This is formalized in the next statement, which is analogous to~\cite[Proposition~2.8]{carrillo2020long}.
\begin{lemma}\label{lem:interaction}
Let $W\in \mathcal{C}^{0}(\TT^d;\RR)$ and let $\Ww$ be the corresponding pairwise interaction energy, \eqref{eq:interaction}. If
$
\tau \geq 4\sum_{k\in \ZZ^d\setminus \{0\}} (\hat W_k)_-
$, 
then $\Ww+\tau \Hh$ is convex.
\end{lemma}
\begin{proof}
The first-variation of $\Ww$ is given by $\Ww'[\mu](x)=\int (W(x-y)+W(y-x))\d\mu(y)=2\int W(x-y)\d\mu(y)$ using the symmetry of $W$. Therefore, for $\rho_1,\rho_2\in \Pp(\TT^d)$ it holds
$$
\Vert \Ww'_-(\rho_2)-\Ww'_-(\rho_1)\Vert_{\rm osc} = 2\Big\Vert \int W_-(\cdot -y)\d(\rho_1(y)-\rho_2(y))\Big\Vert_{\rm osc} \leq 2 \Vert W_-\Vert_{\rm osc} \Vert \rho_2-\rho_1\Vert_{\rm TV}
$$
since for any $f\in \Cc^0(\Omega)$ it holds $|\int f\d(\rho_2-\rho_1)|\leq \frac12 \Vert f\Vert_{\mathrm{osc}} \Vert \rho_1-\rho_2\Vert_{\mathrm{TV}}$. 
Moreover $\Vert W_-\Vert_{\rm osc}\leq 2\sum_{k\in\ZZ^d\setminus \{0\}} (\hat W_k)_-$. Therefore, a computation as in \eqref{eq:towards-smooth} shows that, by posing $C(W_-)=4\sum_{k\in\ZZ^d\setminus \{0\}}(\hat W_k)_-$, it holds
$$
\Big\vert \Ww_-(\rho_2)-\Ww_-(\rho_1)-\int \Ww'_-(\rho_1)(\rho_2-\rho_1)\Big\vert \leq \frac{C(W_-)}{2}\Vert \rho_2-\rho_1\Vert_{\rm TV}^2
$$
so $\Ww_-$ is $(-C(W_-))$-semi convex relative to $\Vert \cdot \Vert_{\rm TV}$. By Pinsker's inequality, $\Hh$ is $1$-strongly convex with respect to the norm $\Vert \cdot \Vert_{\rm TV}$. Therefore, $\Ww_{-}+C(W_-)\Hh$ is convex.
\end{proof}

Consequently, we can apply our results (Theorem~\ref{thm:compact}) and obtain: 

\begin{corollary}\label{cor:interactions}
Let $\Omega=\TT^d$. Let $\Ff=\Gg_0+\Ww+\tau\Hh$, where $\Gg_0$ is convex and satisfies Assumption~\ref{ass:regularity-G-W1}, and $\Ww(\mu)=\int W(y-x)\d\mu(x)\d\mu(y)$ with $W\in \Cc^{1,1}(\TT^d)$. Then, the WGF  $(\mu_t)_t$ of $\Ff$ starting from $\mu_0\in \Pp_2(\Omega)$ (defined in~\eqref{eq:PDE}) is well-defined. Let $(\hat W_k)_{k\in \ZZ^d}$ be the Fourier coefficients of $W$.  If $$\tau>4\sum_{k\in \ZZ^d\setminus\{0\}} (\hat W_k)_-,$$ then $\Ff(\mu_t)-\Ff(\mu^*)$ converges to $0$ exponentially with rate proportional to $ \tau - 4\sum_{k\in \ZZ^d\setminus\{0\}} (\hat W_k)_-$.
\end{corollary}
\begin{proof}
From the expression of the first variation $\Ww'[\mu](x)=2\int W(x-y)\d\mu(y)$, it is immediate to check that $\Ww$ satisfies Assumption~\ref{ass:regularity-G-W1} with $L$ being twice the Lipschitz constant of $\nabla W$.
    By Proposition~\ref{prop:reg-to-convex} and Lemma~\ref{lem:interaction}, the function $\Ff$ satisfies the assumptions of Theorem~\ref{thm:compact} (case $\tau>\tau_c$) and the claim follows.
\end{proof}

\begin{remark}[Comparison to prior works]
The case $\Gg_0=0$ (particles with only pairwise interactions) has been studied in many prior works. A dedicated analysis of the case of the $d$-torus $\TT^d$ can be found in~\cite{carrillo2020long, chayes2010mckean}, where the property of positive definiteness of the kernel is crucial and is referred to as \emph{H-stability}. Notably, \cite[Theorem~1.1]{carrillo2020long} shows exponential convergence to the stationary state $\mu^*$ under a condition comparable to the one in Lemma~\ref{lem:interaction}. Their proof relies on the specific structure of pairwise interactions energies and on the fact that in their setting $\mu^*$ is the uniform distribution. In contrast, Corollary~\ref{cor:interactions} handles more general objectives including a general convex term $\Gg_0\neq 0$ (leading, in particular, to non-uniform minimizers $\mu^*$).
\end{remark}

\begin{remark}[Noncompact case]\label{rem:noncomp_5}
In the case $\Omega=\RR^d$, a similar decomposition of the interaction energies as a difference of convex functions is possible. Let $W$ be the Fourier transform of a finite signed measure $\lambda$. By Bochner's theorem, the Fourier transform  $W_+$ (resp.~$W_-$) of $\lambda_+$ (resp.~$\lambda_-$), the positive (resp.~negative) part of $\lambda$, is a positive definite interaction kernel. The corresponding energy $\Ww_\pm(\mu)=\int W_\pm(y-x)\d\mu(x)\d\mu(y)$ is a convex function on $\Pp_2(\Omega)$. Provided Assumption~\ref{ass:regularity-G} is satisfied, the condition for $\tau$ in Corollary~\ref{cor:interactions} thus becomes $\tau>4\lambda_-(\RR^d\setminus \{0\})$ where $\lambda_-$ denotes the negative part of $\lambda$. We refer to~\cite[Section 3.1]{chen2022uniform} for examples of interaction kernels applicable to our setting.
\end{remark}

\subsection{Approximate Fisher-Information regularization of convex objectives}\label{sec:AFI}
Let $\Gg_0$ be a \emph{convex} function satisfying Assumption~\ref{ass:regularity-G} (or \ref{ass:regularity-G-W1}) and consider the problem of finding an approximate minimizer of $\Gg_0$ using WGFs.

Clearly, one approach is to add $\tau\Hh$ to this function to put ourselves in the framework $\tau>\tau_c=0$ of Theorem~\ref{thm:compact} and guarantee convergence of the WGF. The approximation error incurred by this regularization can be quantified. Denoting $\mu^*_\tau$ the minimizer of $\Gg_0+\tau\Hh$, it holds when $\tau\to0$ (see the proof of Proposition~\ref{prop:AFI}(3)):
\begin{align}\label{eq:approx-entropy}
    \Gg_0(\mu_\tau^*)-\Gg_0(\mu_0^*)=O(\tau\log(1/\tau)),
\end{align} 
where $O(\cdot )$ hides factors depending on $L$ and $d$. We note that, up to the log factor, this rate cannot be improved. This can be seen with an explicit analysis (via Laplace's method) of the case $\Gg(\mu)=\int V\d\mu$, where $V\in \Cc^2(\Omega)$ has a unique minimizer $x^*$ and $\nabla^2 V(x^*) \succ 0$, which gives $\Gg_0(\mu_\tau^*)-\Gg_0(\mu_0^*)=\Theta(\tau)$. 

Let us present another approach which leads, in certain cases, to a smaller approximation error while maintaining the same diffusivity level $\tau$. The idea is to add a \emph{debiasing} term $\Dd(\mu)$ with the following properties: (i) it satisfies  Assumption~\ref{ass:regularity-G-W1}, (ii) it is such that $\Dd+\tau \Hh$ is convex, and (iii) it is equal to $-\tau \Hh$ at leading order when $\tau\to 0$. Such a debiasing term can be derived from the theory of entropic optimal transport:
 
\begin{definition}[Entropic optimal transport]  \label{defi:entropic_ot} Let $\Omega=\TT^d$ and let $c_\tau:\Omega^2\to \RR$ be the smooth function defined by $c_\tau(x,y)=-\tau \log q(\tau,x,y)$, where $q$ is the transition kernel of the standard Brownian motion on the torus. For $\mu,\nu\in \Pp(\Omega)$ and $\tau\geq 0$, let
\begin{align}\label{eq:EOT}
\Tt_\tau(\mu,\nu) = \min_{\gamma \in \Pi(\mu,\nu)} \int c_\tau \d\gamma +\tau \Hh(\gamma|\mu\otimes \nu),
\end{align}
where $\Pi(\mu,\nu)=\{ \gamma \in \Pp(\Omega^2)| \int \gamma(\cdot,\d y)=\mu,\; \int \gamma(\d x,\cdot)=\nu\}$ is the set of transport plans between $\mu$ and $\nu$.
\end{definition}

The next result shows that the regularizer defined as $\mu \mapsto \Tt_\tau(\mu,\mu)+\tau\Hh(\mu)$ satisfies all the desired properties; and allows to gain $1$-order in the approximation error when $\Gg_0$ admits smooth minimizers. It also shows that this regularizer approximates the Fisher Information $\Ii(\mu)=\int  \vert \nabla \log(\mu)\vert^2\d\mu$ at leading order when $\tau\to 0$.

\begin{proposition}\label{prop:AFI}
Let $\Omega=\TT^d$. Let $\Gg_0:\Pp(\Omega)\to \RR$ be a \emph{convex} function satisfying Assumption~\ref{ass:regularity-G-W1} and let $\Dd_\tau(\mu)=\Tt_\tau(\mu,\mu)$ (see Definition~\ref{defi:entropic_ot}). Then:
\begin{enumerate}
\item (Convergence of the WGF) $\Dd_\tau$ satisfies Assumption~\ref{ass:regularity-G-W1} and $\Dd_\tau+\tau \Hh$ is convex. In particular, the convergence guarantee of Theorem~\ref{thm:main} applies and the WGF of $\Gg_0+\Dd_\tau+\tau\Hh$ starting from any $\mu_0\in \Pp_2(\Omega)$ converges at a rate $O(1/t)$.
\item (Approximation of Fisher Information) It holds $0\leq \Dd_\tau(\mu)+\tau\Hh(\mu)\leq \frac{\tau^2}{8}\Ii(\mu)$, where $\Ii(\mu)=\int  \vert \nabla \log(\mu)\vert^2\d\mu$. Moreover, if $\Ii(\mu)<+\infty$, then $\Dd_\tau(\mu) + \tau\Hh(\mu)=\frac{\tau^2}{8}\Ii(\mu)+o(\tau^2)$ as $\tau\to 0$.
\item (Approximation error) If we denote $\bar \mu^*_\tau$ a minimizer of $\Gg_0+\Dd_\tau+\tau\Hh$, it holds 
$$
\Gg(\bar \mu^*_\tau)-\Gg(\mu^*_0)=O(\tau).
$$
If, moreover, a minimizer $\mu_0^*$ of $\Gg_0$ satisfies $\Ii(\mu_0^*)<\infty$, then \[\Gg(\bar \mu^*_\tau)-\Gg(\mu^*_0)\leq \frac{\tau^2}{8} \Ii(\mu_0^*).
\]
\end{enumerate}
\end{proposition}

\begin{remark}
In Proposition~\ref{prop:AFI}(3), we are only able to show that this regularizer \emph{does not worsen} the worst-case approximation error compared to \eqref{eq:approx-entropy}. (There is an apparent gain of a log factor; however, this log factor might not be tight.) If $\Ii(\mu_0^*)<+\infty$, we improve this error bound to $O(\tau^2)$. However one must be careful when comparing upper-bounds; as it is possible that the actual approximation rates for both $\mu_\tau^*$ and $\bar \mu_\tau^*$ are faster (perhaps under more specific assumptions). This subtlety is apparent in the Euclidean setting:  if we let $x^*_\tau = \arg\min_{x\in \RR^d} f(x)+\tau g(x)$ with $f$ admitting a unique minimizer $x_0^*$ with $\nabla^2 f[x_0^*]\succ 0$, and $f,g$ convex,  and smooth in a neighborhood $x_0^*$, then an application of the implicit function theorem shows that 
$$
f(x^*_\tau) - f(x_0^*)=\frac12 \tau ^2 \nabla g(x_0^*) [\nabla^2 f(x_0^*)]^{-1} \nabla g(x_0^*)+o(\tau^2)
$$
and therefore is \emph{quadratic} in $\tau$ locally. Such a behavior, not captured by our bounds, cannot be excluded in our setting in general.
\end{remark}
Let us also remark that the potential gain in approximation error due to using $\Dd_\tau+\tau\Hh$ as a regularizer comes at the cost of the loss of strong convexity, and that while the diffusivity level is $\tau$ in both cases, the two-sided bounds on the density (which appear in the convergence rates) might worsen when adding $\Dd_\tau$ to the objective.

\begin{proof}[Proof of Proposition~\ref{prop:AFI}]
Let us first recall a few facts on entropic optimal transport (EOT), which can be found, e.g.~in~\cite{genevay2019sample, feydy2019interpolating,carlier2024displacement}, before proving these points separately.
\paragraph{Reminders on EOT.} The dual formulation of EOT states that
\begin{align}\label{eq:EOT-dual}
\Tt_\tau(\mu,\nu) = \max_{\varphi,\psi} \int \varphi \d\mu +\int \psi\d\nu +\tau \int e^{(\varphi(x)+\psi(y)-c_\tau(x,y))/\tau}\d\mu(x)\d\nu(y) -\tau,
\end{align}
where the maximum is over continuous functions on the torus.  It is known that the maximum is attained. Moreover, among the potentially non-unique maximizing pairs, there is a unique one (up to the invariance $(\varphi,\psi)\mapsto (\varphi+c,\psi-c)$, $c\in \RR$) that solves for all $x,y\in \Omega^2$ 
\begin{align}\label{eq:schrodinger-system}
    \varphi(x) &= -\tau \log \int e^{(\psi(y)-c_\tau(x,y))/\tau}\d\nu(y),&
    \psi(y) &= -\tau \log \int e^{(\varphi(x)-c_\tau(x,y))/\tau}\d\mu(x).
\end{align}
These functions, which we will denote by $\varphi[\mu,\nu],\psi[\mu,\nu]$ and refer to as the \emph{EOT potentials}, are the first-variation of $\Tt_\tau$ w.r.t.\ $\mu$ and $\nu$, respectively. It can be directly checked from \eqref{eq:schrodinger-system} that their gradients are Lipschitz continuous (with the same Lipschitz constant as $c_\tau$ in one of its variables), and it was proved in~\cite{carlier2024displacement} that there exists $C>0$ such that for any $\mu,\mu',\nu,\nu'\in \Pp(\Omega)$,
\begin{align}
\Vert \nabla \varphi[\mu,\nu] - \nabla \varphi[\mu',\nu']\Vert_\infty + \Vert \nabla \psi[\mu,\nu] -\nabla \psi[\mu',\nu']\Vert_\infty \leq C (W_2(\mu,\mu')+W_2(\nu,\nu')).
\end{align}
In fact, the bound with $W_1$ instead of $W_2$ also holds, by replacing the use of Cauchy-Schwartz in their Eq.~(20) with an $L^1$-$L^\infty$ H\"older inequality; see also~\cite{deligiannidis2024quantitative}.

\paragraph{Proof of 1.} Since $\Dd_\tau$ is the composition of the diagonal map $\mu\mapsto (\mu,\mu)$ with $\Tt_\tau$, its first-variation is $\varphi[\mu,\mu]+\psi[\mu,\mu]$ (note that due to the symmetries of the problem, we have in this case $\varphi[\mu,\mu]=\psi[\mu,\mu]+c$ for some $c\in \RR$). By the previous considerations, this shows that $\Dd_\tau$ satisfies Assumption~\ref{ass:regularity-G-W1}. Now $\Dd_\tau$ is \emph{not} convex; and is in fact concave~\cite[Proposition~4]{feydy2019interpolating}. Therefore, the convexity of $\Dd_\tau+\tau \Hh$ must rely on a more subtle interaction between $\Dd_\tau$ and $\tau\Hh$, and can be argued as follows:

Using the expression of $\Hh$ and the fact that $\gamma\in \Pi(\mu,\nu)$ (see details for computations of this nature in~\cite[Lemma 1.5]{marino2020optimal}), we have if $\Hh(\mu),\, \Hh(\nu)<+\infty$,
\begin{align*}
\Tt_\tau(\mu,\nu)+\tau\Hh(\nu) &= \min_{\gamma \in \Pi(\mu,\nu)} \int c_\tau \d\gamma +\tau \Hh(\gamma|\mu\otimes \d y)\\
&=\max_{\varphi,\psi} \int \varphi \d\mu +\int \psi\d\nu +\tau \int e^{(\varphi(x)+\psi(y)-c_\tau(x,y))/\tau}\d\mu(x)\d y -\tau
\end{align*}
where $\d y$ denotes the Lebesgue measure on $\Omega$, and the second expression follows by a direct adaptation of the duality argument leading to~\eqref{eq:EOT-dual}. Hence,  $\Tt_\tau(\mu,\nu)+\tau\Hh(\nu)$ can be expressed as a maximum of linear functions of $(\mu,\nu)$, and it is therefore convex (jointly in $(\mu,\nu)$). This is enough to deduce that $\Dd_\tau+\tau\Hh$ is convex. For later use, let us also remark that an analogous argument shows that $\Tt_\tau(\mu,\nu)+\tau\Hh(\mu)$ is convex, and by convex combinations, any function of the form $\Tt_\tau(\mu,\nu)+\theta \tau\Hh(\mu)+(1-\theta) \tau\Hh(\nu)$ for $\theta\in [0,1]$ is convex in $(\mu,\nu)$.

\paragraph{Proof of 2.} Let us now argue why this regularizer can be interpreted as an approximate Fisher-Information regularization. The dynamic formulation of EOT~\cite{chen2016relation, gentil2017analogy, leger2019geometric} states that if $\Hh(\mu),\Hh(\nu)<+\infty$ 
\begin{align}\label{eq:EOT-DYN}
\Tt_\tau(\mu,\nu)+\frac{\tau}{2}(\Hh(\mu)+\Hh(\nu)) = \inf \int_0^1 \int_\Omega \left(\frac12 \vert v_t\vert^2 +\frac{\tau^2}{8} \vert \nabla \log \rho_t\vert^2\right)\rho_t\d x\d t,
\end{align}
where the infimum runs over $(\rho_t)_{t\in [0,1]}$ absolutely continuous paths in $\Pp_2(\Omega)$, and $(v_t)_{t\in [0,1]}$ family of velocity fields satisfying  $\partial_t \rho_t = - {\rm div}\,  (\rho_t v_t)$ and $(\rho_0,\rho_1)=(\mu,\nu)$. Since $\rho_t=\mu$ and $v_t=0$ is admissible in \eqref{eq:EOT-DYN} when $\mu=\nu$, we directly get
\begin{align}\label{eq:bound-FI}
0\leq \Dd_\tau(\mu)+\tau \Hh(\mu) \leq \frac{\tau^2}{8} \Ii(\mu),
\end{align}
where, we recall, $\Ii(\mu)=\int \vert \nabla \log\mu\vert^2\d\mu$ is the Fisher Information of $\mu$. We can also argue as in~\cite[Theorem 1]{chizat2020faster} or~\cite[Theorem~1.6]{conforti2021formula} to obtain that, if $\Ii(\mu)<+\infty$, then
$
\Dd_\tau(\mu)+\tau \Hh(\mu)=\frac{\tau^2}{8} \Ii(\mu) +o(\tau^2).
$
\paragraph{Proof of 3.} Let us now study the approximation errors, starting from the one with entropic regularization written in~\eqref{eq:approx-entropy}. Let us denote by $\mu^*_\tau$ the minimizer of $\Gg_0+\tau\Hh$ and by $\bar \mu^*_\tau$ the minimizer of $\Gg_0+\Dd_\tau+\tau\Hh$. Since the Lebesgue measure $\lambda$ on the torus is a probability measure, $\Hh(\mu)=\Hh(\mu|\lambda)\geq 0$. Therefore, by optimality of $\mu^*_\tau$ it holds
\begin{align}\label{eq:starting-approx}
    \Gg_0(\mu^*_\tau)\leq \Gg_0(\mu^*_\tau)+\tau \Hh(\mu^*_\tau) \leq \Gg_0(\mu^*_0 \ast q_t)+\tau \Hh(\mu^*_0 \ast q_t)
\end{align}
where $q_t=q(t,0,x)$ is, as before, the transition kernel  of the Brownian motion on the torus, used here as a mollifier.  Posing $\mu_s=(1-s)\mu^*_0 +s\mu^*_0 \ast q_t$ (as in~\eqref{eq:towards-smooth})  it holds:
\begin{multline*}
    \Gg_0(\mu^*_0 \ast q_t)-\Gg_0(\mu^*_0) = \int \Gg'[\mu_0^*]\d (\mu^*_0 \ast q_t-\mu_0^*)+\int_0^1\int (\Gg'[\mu_s]-\Gg'[\mu_0^*])\d (\mu^*_0 \ast q_t-\mu_0^*)\d s.
\end{multline*}
By the first-order optimality condition, we have $\Gg'[\mu_0^*](x)=\min \Gg'[\mu_0^*]$ for $\mu_0^*$-almost every $x$. Since $\Gg'[\mu_0^*]$ has a Lipschitz gradient, and therefore at most quadratic growth around its minimizers, the first term in the right-hand side is in $O(t)$ as $t\to 0$. Also by Assumption~\ref{ass:regularity-G-W1} we have $\Vert\nabla \Gg'[\mu_s]-\nabla \Gg'[\mu_0^*]\Vert_{\infty}\leq LW_1(\mu_s,\mu_0^*)$. Therefore the second term is absolutely bounded by 
$$LW_1(\mu^*_0\ast q_t, \mu_0^*)\int_0^1 W_1(\mu_s,\mu_0^*)\d s\leq LW_1(\mu^*_0\ast q_t, \mu_0^*)^2\leq LW_2(\mu^*_0\ast q_t, \mu_0^*)^2=O(t),
$$ where the last bound is obtained by considering the coupling $(X+B_t,X)$ for $X\sim \mu_0^*$ and $B_t$ a Brownian motion at time $t$ (independent from $X$). Therefore, we have $\Gg_0(\mu^*_0 \ast q_t)-\Gg_0(\mu^*_0)=O(t)$. It remains to bound $\tau \Hh(\mu^*_0 \ast q_t)$. By Jensen's inequality
$$
\Hh(\mu^*_0 \ast q_t) \leq \Hh(q_t) \leq \Hh(g_t) = O(\log(1/t))
$$
where $g_t$ is the transition kernel of the Brownian motion on $\RR^d$ at time $t$ (for the last inequality, this uses that $q_t=T_\# g_t$ where $T$ maps $x\in \RR^d$ to its representative in $\TT^d$). Therefore, coming back to \eqref{eq:starting-approx} with the estimates obtained so far, we have
$$
\Gg_0(\mu^*_0 \ast q_t)-\Gg_0(\mu^*_0)  = O(t+ \tau \log(1/t)).
$$
The upper bound~\eqref{eq:approx-entropy} follows by taking $t=\tau$.
We proceed similarly to upper bound the suboptimality gap of $\bar \mu^*_\tau$, but we use \eqref{eq:bound-FI} to get
$$
\Gg_0(\bar \mu^*_\tau)-\Gg_0(\mu^*_0)\leq \Gg_0( \mu^*_0 \ast q_t)-\Gg_0(\mu^*_0)+\frac{\tau^2}{8} \Ii(\mu^*_0 \ast q_t).
$$
By Jensen's inequality again, we have
$$
\Ii(\mu^*_0 \ast q_t)\leq \Ii(q_t) \leq \Ii(g_t)=d/t.
$$
Therefore we have
$$
\Gg_0(\bar \mu^*_\tau)-\Gg_0(\mu^*_0)\leq O(t+\tau^2/t).
$$
After optimizing on the parameter $t$, we obtain $t=\tau$ and the bound is $O(\tau)$.
Finally, the last claim follows from the following inequality, which is a consequence of the optimality of $\bar \mu^*_\tau$ and~\eqref{eq:bound-FI}:
\begin{equation*}
\Gg_0(\bar \mu_\tau)\leq \Gg_0(\mu_0^*) + \Dd_\tau(\mu_0^*)+\tau \Hh(\mu^*_0) \leq \Gg_0(\mu_0^*) + \frac{\tau^2}{8}\Ii(\mu^*_0).\qedhere
\end{equation*}
\end{proof}

\subsection{Diffusion in path space for trajectory inference}\label{sec:trajectory-inference}
Finally, let us present a setting where our guarantees apply without the need for extra regularization. 

Let $\Omega=\TT^d$ and let $(\rho_t)_{t\in [0,1]}$ be a path in the space of probability measures. Suppose that one disposes of $(\hat \rho_{t_0},\dots,\hat \rho_{t_T})\in \Pp(\Omega)^{T+1}$ which are noisy or incomplete observations of this path at $T+1$ regular times $0=t_0<\dots <t_T=1$ (so that $t_{i+1}-t_i=1/T$). The problem of trajectory inference is to estimate $(\rho_t)_{t\in [0,1]}$ based on these observations. 

An estimator for trajectory inference, formulated as an optimization problem in path-space, was proposed in~\cite{lavenant2021toward}. Consider the following smoothed log-likelihood as a data-fitting term, for $\rho,\hat \rho\in \Pp(\Omega)$:
$$
{\rm Fit}(\rho|\hat \rho) = -\int \log(\rho\ast q_\sigma)\d\hat \rho 
$$
where $q_\sigma$ is the transition kernel of the Brownian at time $\sigma>0$ (here $\sigma$ is a hyperparameter). Clearly, $\rho \mapsto {\rm Fit}(\rho|\hat \rho) $ is convex and it is not difficult to show that it also satisfies Assumption~\ref{ass:regularity-G-W1} thanks to the convolution operation. 

Let $\Cc([0,1];\Omega)$ be the space of continuous paths in $\Omega$ endowed with the sup-norm. The estimator in~\cite{lavenant2021toward} is defined as the minimizer of the following functional defined for $R\in \Pp(\Cc([0,1];\Omega))$ as
\begin{align}\label{eq:path-space}
\frac1{T+1} \sum_{i=0}^T \mathrm{Fit}(R_{t_i}|\hat \rho_{t_i}) +\tau \Hh(R|\WW^\tau)
\end{align}
where $\WW^\tau\in \Pp(\Cc([0,1];\Omega))$ is the law of the reversible Brownian motion at diffusivity $\tau$ and $R_t$ denotes the marginal of $R$ at time $t$. The first term encourages the process $R$ to have time marginals consistent with the observations, while the second term acts as a regularizer, favoring processes which do not deviate too much from  a Brownian motion. A key contribution of~\cite[Theorem~1.2]{lavenant2021toward} is to show that, for a large class of drift-diffusion processes as a ground truth, this estimator consistently recovers its law if $T\to +\infty$ and if $\hat \rho_{t_i}$ are empirical distributions (with at least $1$ sample each).

As remarked in~\cite{lavenant2021toward}, this  problem can be equivalently reformulated as a minimization over a family of $T+1$ probability measures $\boldsymbol{\mu} = (\mu_0,\dots,\mu_T)\in \Pp(\Omega)^{T+1}$ thanks to the Markovian property of the reference process $\WW^\tau$. Specifically, the minimizer of~\eqref{eq:path-space} can be explicitely reconstructed from the minimizer of 
\begin{align}\label{eq:path-function}
\Ff(\boldsymbol{\mu}) = \underbrace{\frac1{T+1}\sum_{i=0}^T \mathrm{Fit}(\mu_{i}|\hat \rho_{t_i}) + T \sum_{i=0}^{T-1} \Tt_{\tau/T}(\mu_i,\mu_{i+1})}_{\Gg}+\tau \sum_{i=0}^T \Hh(\mu_i)
\end{align}
(with $\Tt_\tau$ defined in~\eqref{eq:EOT}) and the minima achieved are the same for both formulations   (and for the corresponding minimizers we have $R_{t_i} = \mu_{t_i}$).

This precise formulation is given in~\cite[Theorem~3.1]{chizat2022trajectory}, where the WGF of $\Ff$ was proposed  as an optimization method to compute the estimator. However, in~\cite{chizat2022trajectory}, the convergence guarantees do not apply to~\eqref{eq:path-function}, but to a different problem with an extra regularization.  Here, we remark that our general results directly apply to this setting, thereby providing a theoretical guarantee to compute the original estimator.
\begin{proposition}
Let $\Omega=\TT^d$ and let $\Gg:\Pp_2(\Omega)^{T+1}\to \RR$ be the function defined in~\eqref{eq:path-function}. Then $\Gg$  satisfies Assumption~\ref{ass:regularity-G-W1} and $\Ff$ is $\tau/(T+1)$-strongly convex relative to $\boldsymbol{\mu}\mapsto \sum_i \Hh(\mu_i)$.  In particular, the WGF of $\Ff$ starting from any $\boldsymbol{\mu}_0\in \Pp_2(\Omega)^{T+1}$ converges at the exponential rate given by Theorem~\ref{thm:compact}. 
\end{proposition}

\begin{proof}
First, let us note that, \emph{mutatis mutandis}, the analysis presented in this paper extends to the case of a function of a family of probability measures, which is the setting considered here.
 
The regularity of $\Gg$ follows from that of $\Tt_\tau$ (given in the proof of Proposition~\ref{prop:AFI}) and of the regularized log-likelihood. We recall also from the proof of Proposition~\ref{prop:AFI} that for any $\theta\in [0,1]$, functions of the form 
$$
(\mu,\nu)\mapsto \Tt_\tau(\mu,\nu)+\theta\tau\Hh(\mu)+(1-\theta)\tau \Hh(\nu)
$$
are convex. Using the decomposition 
$$
\frac{T}{T+1}\sum_{i=0}^T\Hh(\mu_i) = \sum_{i=0}^{T-1} \left[\left(1-\frac{i+1}{T+1}\right)\Hh(\mu_i) +\frac{i+1}{T+1}\Hh(\mu_{i+1})\right],
$$
it follows that 
$$
T \sum_{i=0}^{T-1} \Tt_{\tau/T}(\mu_i,\mu_{i+1})+\frac{\tau T}{T+1} \sum_{i=0}^T \Hh(\mu_i)
$$
is jointly convex in $(\mu_0,\dots,\mu_T)$, and therefore $\Ff$ is $\tau/(T+1)$ strongly convex relative to $\boldsymbol{\mu}\mapsto \sum_i \Hh(\mu_i)$ (the quantity $\tau/(T+1)$ here plays the same role in the convergence rate as $\tau-\tau_c$ in the case of a single probability measure).
\end{proof}

\paragraph{A variant at the critical diffusivity} The entropic regularization in \eqref{eq:path-space}, in addition to favoring regularity in time of the recovered process, also encourages its time marginals to be diffused, because the time marginals of the reversible Brownian motion are uniform. This second effect can be problematic in situations where the ground truth process has time marginals which are far from uniform. This behavior can be partially alleviated  by considering instead the  objective (recall \eqref{eq:path-function})
\begin{align}\label{eq:path-function-bis}
\tilde \Ff({\boldsymbol \mu}) &= \Ff({\boldsymbol \mu}) - \frac{\tau}{2}(\Hh(\mu_0)-\Hh(\mu_T))\\
&=\Gg({\boldsymbol \mu})+ \frac{\tau}{2} \sum_{i=0}^T \Hh(\mu_i) + \frac{\tau}{2} \sum_{i=1}^{T-1} \Hh(\mu_i).
\end{align}

An interesting aspect of this regularizer is that it behaves as an approximation, up to a term in $O(\tau^2)$, of the Benamou-Brenier energy~\cite{benamou2000computational} of the family of measures $\boldsymbol{\mu}$, in a sense made precise in the next statement. (The Benamou-Brenier energy is the value of the minimum in~\eqref{eq:critical-path} below when $\tau=0$.)
\begin{proposition}\label{prop:trajectory-bis-regularrization} 
It holds
\begin{equation}\label{eq:critical-path}
    \tilde \Ff({\boldsymbol \mu}) = \frac1{T+1}\sum_{i=0}^T \mathrm{Fit}(\mu_{i}|\hat \rho_{t_i}) +\min \int_0^1 \Big(\frac12 \vert v_t\vert_2^2 +\frac{\tau^2}{8} \vert \nabla \log \rho_t\vert^2\Big)\d \rho_t \d t
\end{equation}
where the minimum runs over $(\rho_t)_{t\in [0,1]}$ absolutely continuous path in $\Pp_2(\Omega)$ and $(v_t)_{t\in [0,1]}$ family of velocity fields satisfying  $\partial_t \rho_t = - {\rm div}\,  (\rho_t v_t)$ and $(\rho_{t_0},\dots,\rho_{t_T})=(\mu_0,\dots,\mu_T)$.
\end{proposition}
\begin{proof}
    In the notations of the proof of Proposition~\ref{prop:trajectory-bis-convergence}, by \eqref{eq:EOT-DYN} it holds
    \begin{align*}
    A_i &= T \min \int_0^1 \Big(\frac12 \vert v_t\vert^2 +\frac{\tau^2}{8} \vert \nabla \log \rho_t\vert^2\Big)\d \rho_t \d t\\
    &=\min \int_{t_i}^{t_{i}+1/T} \Big(\frac12 \vert \tilde v_s\vert_2^2 +\frac{\tau^2}{8} \vert \nabla \log \tilde \rho_s\vert^2\Big)\d \rho_t \d s.
    \end{align*}
    where $(\rho_t,v_t)_{t\in [0,1]}$ solves the continuity equation $\partial_t \rho_t=-{\rm div}\,  (v_t\rho_t)$ and $(\rho_0,\rho_1)=(\mu_i, \mu_{i+1})$ and one goes from the first to the second line by a change of time variable $s=t_i +t/T$.
    Since $\tilde \Ff({\boldsymbol \mu}) = \frac1{T+1}\sum_{i=0}^T \mathrm{Fit}(\mu_{i}|\hat \rho_{t_i}) + \sum_{i=0}^{T-1} A_i$, the result follows by glueing the $T$ integral terms.
\end{proof}

Finally, our results show that this objective $\tilde \Ff$ lends itself well to optimization via WGF:
\begin{proposition}\label{prop:trajectory-bis-convergence}
The function $\tilde \Ff$ is convex. In particular, by Theorem~\ref{thm:compact}, the WGF of $\tilde \Ff$ starting from any $\boldsymbol{\mu}_0\in \Pp_2(\Omega)^{T+1}$ converges at the rate $O(1/t)$ given by Theorem~\ref{thm:compact}. 
\end{proposition}
\begin{proof}
This objective can be written as $\tilde \Ff({\boldsymbol \mu}) = \frac1{T+1}\sum_{i=0}^T \mathrm{Fit}(\mu_{i}|\hat \rho_{t_i}) + \sum_{i=0}^{T-1} A_i$ with
    $$
   A_i \coloneqq T\cdot \Tt_{\tau/T}(\mu_i,\mu_{i+1}) +\frac{\tau}{2} (\Hh(\mu_i)+\Hh(\mu_{i+1})).
    $$
    The data fitting terms are convex and moreover the terms $A_i$ are convex functions of $(\mu_i,\mu_{i+1})$, as seen in the proof of Proposition~\ref{prop:AFI}. Therefore $\tilde \Ff$ is convex (it is however not strongly convex relative to $\Hh$ \emph{a priori}). Then one can apply Theorem~\ref{thm:compact} since the diffusivity for each of the arguments $\mu_i$ is lower bounded by $\tau/2$. 
\end{proof}

\subsection*{Acknowledgements:} 
We would like to thank Lucio Galeati for useful discussions on transition kernel estimates. 

MC is supported by the Swiss State Secretariat for Education, Research and Innovation (SERI) under contract number MB22.00034 through the project TENSE. 

XF was supported by the Swiss National Science Foundation (SNF grant
PZ00P2\_208930) and by the AEI project PID2021-125021NA-I00 (Spain).
 
\printbibliography

\end{document}

%% file: shortcuts.tex
\newcommand{\RR}{\mathbb{R}}

\newcommand{\ZZ}{\mathbb{Z}}
\renewcommand{\PP}{\mathbb{P}}
\newcommand{\WW}{\mathbb{W}}

\newcommand{\Cc}{\mathcal{C}}
\newcommand{\Dd}{\mathcal{D}}

\newcommand{\Ff}{\mathcal{F}}
\newcommand{\Gg}{\mathcal{G}}
\newcommand{\Hh}{\mathcal{H}}
\newcommand{\Ii}{\mathcal{I}}

\newcommand{\Nn}{\mathcal{N}}

\newcommand{\Pp}{\mathcal{P}}

\newcommand{\Tt}{\mathcal{T}}

\newcommand{\Ww}{\mathcal{W}}
\newcommand{\Xx}{\mathcal{X}}

\newcommand{\E}{\mathbf{E}}

\renewcommand{\d}{\mathrm{d}}

\def\ones{\mathds{1}}

\newcommand{\eps}{\varepsilon}

\DeclareMathOperator{\dist}{dist}

\newtheorem{theorem}{Theorem}[section]
\newtheorem{proposition}[theorem]{Proposition}
\newtheorem{lemma}[theorem]{Lemma}
\newtheorem{corollary}[theorem]{Corollary}
\newtheorem{remark}[theorem]{Remark}

\newtheorem{definition}[theorem]{Definition}
\newtheorem{assumption}{Assumption}

\newenvironment{assumptionprime}[1][A]
  {\renewcommand{\theassumption}{#1$'$}%
   \begin{assumption}}
  {\end{assumption}}